\documentclass[reqno,10pt]{amsart}

\usepackage{amssymb,psfrag,graphicx}
\usepackage[utf8]{inputenc}
\usepackage{enumerate}
\usepackage{color}

\usepackage{marginnote}

\numberwithin{equation}{section}

\theoremstyle{plain}
\newtheorem{thm}{Theorem}[section]
\newtheorem{lem}{Lemma}[section]
\newtheorem{prop}{Proposition}[section]
\newtheorem{cor}{Corollary}[section]

\theoremstyle{definition}
\newtheorem{rmk}{Remark}[section]

\newcommand{\be}{\begin{equation}}
\newcommand{\ee}{\end{equation}}
\newcommand{\bea}{\begin{eqnarray}}
\newcommand{\eea}{\end{eqnarray}}
\newcommand{\beas}{\begin{eqnarray*}}
\newcommand{\eeas}{\end{eqnarray*}}
\newcommand{\rd}{{\text{\rm d}}}
\DeclareMathOperator{\Exp}{e}

\newcommand{\f}{\mathbf{f}}

\newcommand{\bp}{\mathbf{p}}
\newcommand{\bq}{\mathbf{q}}
\newcommand{\bu}{\mathbf{u}}
\newcommand{\bv}{\mathbf{v}}
\newcommand{\bw}{\mathbf{w}}
\newcommand{\bx}{\mathbf{x}}
\newcommand{\by}{\mathbf{y}}
\newcommand{\bg}{\mathbf{g}}

\newcommand{\mC}{\mathcal{C}}

\newcommand{\mL}{\mathcal{L}}

\newcommand{\mV}{\mathcal{V}}

\begin{document}

\title[Postprocessing Galerkin method in data assimilation]{Postprocessing Galerkin method applied to a data assimilation algorithm: a uniform in time error estimate}

\date{December 21, 2016}


\author[C. F. Mondaini]{Cecilia F. Mondaini}
\address{\textnormal{(Cecilia F. Mondaini)} Department of Mathematics\\
Texas A\&M University\\ College Station, TX 77843, USA.}
\email[C. F. Mondaini] {cfmondaini@math.tamu.edu}

\author[E. S. Titi]{Edriss S. Titi}
\address{\textnormal{(Edriss S. Titi)} Department of Mathematics\\
Texas A\&M University\\ College Station, TX 77843, USA. {\bf ALSO}, Science Program, Texas A\&M University at Qatar, Doha, Qatar.}
\email[E. S. Titi] {titi@math.tamu.edu}


\begin{abstract}
We apply the Postprocessing Galerkin method to a recently introduced continuous data assimilation (downscaling) algorithm for obtaining a numerical approximation of the solution of the two-dimensional Navier-Stokes equations corresponding to given measurements from a coarse spatial mesh. Under suitable conditions on the relaxation (nudging) parameter, the resolution of the coarse spatial mesh and the resolution of the numerical scheme, we obtain uniform in time estimates for the error between the numerical approximation given by the Postprocessing Galerkin method and the reference solution corresponding to the measurements. Our results are valid for a large class of interpolant operators, including low Fourier modes and local averages over finite volume elements. Notably, we use here the 2D Navier-Stokes equations as a paradigm, but our results apply equally to other evolution equations, such as the Boussinesq system of B\'enard convection and other oceanic and atmospheric circulation models.
\end{abstract}

\subjclass[2010]{35Q30, 37L65, 65M15, 65M70, 76B75, 93C20}

\keywords{data assimilation, downscaling, nudging, feedback control, two-dimensional Navier-Stokes equations, Galerkin method, postprocessing}

\maketitle

\section{Introduction}

Forecast models attempt to capture the future behavior of a real physical system by using only theoretical arguments. The purpose of data assimilation algorithms is to combine a forecast model with observational data in order to produce an even better approximation of reality. Our goal in this work is to investigate the continuous data assimilation algorithm introduced in the recent work \cite{AzouaniOlsonTiti2014} from a numerical analysis viewpoint, by providing an analytical estimate for the error obtained when using a spatial discretization scheme given by the Postprocessing Galerkin method \cite{GarciaNovoTiti1998,GarciaNovoTiti1999}.

The continuous data assimilation algorithm introduced in \cite{AzouaniOlsonTiti2014} was inspired by ideas from feedback control theory (see, e.g., \cite{AzouaniTiti2014,GhilHalemAtlas1978,GhilShkollerYangarber1977,Luenberger1971,LunasinTiti,Nijmeijer2001,Thau1973} and references therein) and consists in a nudging type approach that is applicable to a large class of dissipative evolution equations. In \cite{AzouaniOlsonTiti2014}, the algorithm is illustrated for the 2D Navier-Stokes equations and under the assumptions of continuous in time and error-free measurements. Further works extended this approach to more general situations, such as continuous in time data assimilation with stochastically noisy data \cite{BessaihOlsonTiti2015} and discrete in time observations with systematic errors \cite{FoiasMondainiTiti2016} (see also \cite{HaydenOlsonTiti2011}). Moreover, noisy observations were also considered in \cite{BlomkerLawStuartZygalakis2013,LawShuklaStuart2014} in the context of the 3DVAR filtering method.
 Applications of this algorithm to the physically important context of incomplete observations, as described, e.g., in \cite{CharneyHalemJastrow1969}, were done in the works
\cite{FarhatJollyTiti2015}-\cite{FarhatLunasinTiti2016-CharneyConj}.

Here, we also consider, as a paradigm, the forecast model given by the two-dimensional Navier-Stokes equations,
\be\label{NSEeqs}
 \frac{\partial \bu}{\partial t} - \nu \Delta \bu + (\bu \cdot \nabla) \bu + \nabla p = \f, \quad \nabla \cdot \bu = 0,
\ee
where $\bu = (u_1,u_2)$ and $p$ are the unknowns and represent the velocity vector field and the pressure, respectively; while $\nu > 0$ and $\f$ are given and denote the kinematic viscosity parameter and the density of volume body forces, respectively. We assume that $\bu$ and $p$ are functions of a spatial variable $\bx$ and a time variable $t$, with $\bx$ varying in a set $\Omega \subset \mathbb{R}^2$ and $t$ varying in the time interval $[t_0,\infty)$. For simplicity, $\f$ is assumed to be time-independent, although similar results are valid for a time-dependent $\f$ whose $L^2$-norm is uniformly bounded in time.

Our reference solution, whose exact value is unknown, is assumed to be a solution $\bu$ of \eqref{NSEeqs}. The given measurements, corresponding to $\bu$, are observed from a coarse spatial mesh and are assumed, for simplicity, to be continuous in time and error-free, as in \cite{AzouaniOlsonTiti2014}. We denote the operator used for interpolating these measurements in space by $I_h$, where $h$ denotes the resolution of the coarse spatial mesh of the observed measurements. Thus, the interpolated measurements are represented by $I_h(\bu)$. Since the initial condition $\bu(t_0)$ for $\bu$ is missing, one cannot compute $\bu$ by integrating \eqref{NSEeqs} directly. The idea consists then in recovering the exact value of the reference solution $\bu$ by using the given measurements, $I_h(\bu)$, through an approximate model. In other words, our purpose here is to provide a downscaling algorithm for recovering the fine scales of $\bu$ from the coarse scale
measurements $I_h(\bu)$.

In \cite{AzouaniOlsonTiti2014}, this is done by seeking for an approximate solution $\bv = (v_1,v_2)$ satisfying the following system, for $(\bx,t) \in \Omega \times I \subset \mathbb{R}^2 \times \mathbb{R}$,
\be\label{DataAssAlg}
 \frac{\partial \bv}{\partial t} - \nu \Delta \bv + (\bv \cdot \nabla) \bv + \nabla \pi = \f - \beta(I_h(\bv) - I_h(\bu)), \quad \nabla \cdot \bv = 0,
\ee
where the unknown $\pi$ is the pressure of the approximate flow $\bv$; $\nu > 0$ and $\f$ are the same viscosity parameter and forcing term from \eqref{NSEeqs}, respectively; and $\beta$ is the relaxation (nudging) parameter. The second term in the right-hand side of the first equation in \eqref{DataAssAlg} is called the feedback control term and its role is to force (or nudge) the coarse spatial scales of the approximating solution $\bv$ towards the coarse spatial scales of the reference solution $\bu$, which is done by suitably tuning the relaxation parameter $\beta$. In \cite[Theorems 1 and 2]{AzouaniOlsonTiti2014}, the authors prove that, provided $\beta$ is large enough and $h$ is sufficiently small, both depending on the physical parameters, the approximate solution $\bv$ of \eqref{DataAssAlg}, corresponding to an arbitrary initial data $\bv_0$, converges, exponentially in time, to the reference solution $\bu$
of \eqref{NSEeqs}.

The aim of this paper is to obtain a numerical approximation of $\bv$, which is done here by using the Postprocessing Galerkin method \cite{GarciaNovoTiti1998,GarciaNovoTiti1999}, and thus indirectly approximate the reference solution $\bu$. First, let us rewrite the system of equations \eqref{DataAssAlg} in the following equivalent functional form
\be\label{DataAssAlgFunc}
\frac{\rd \bv}{\rd t} + \nu A \bv + B(\bv,\bv) = \bg - \beta P_{\sigma}(I_h(\bv) - I_h(\bu)),
\ee
where $P_\sigma$ is the orthogonal projection of $(L^2(\Omega))^2$ onto the phase space $H$ associated to \eqref{DataAssAlg}, which is endowed with the norm of $(L^2(\Omega))^2$, $|\cdot|_{L^2}$; $A = -P_\sigma\Delta$ is the Stokes operator; $B(\bv,\bv) = P_\sigma [(\bv \cdot \nabla)\bv]$, a bilinear operator (see section \ref{secPreliminaries} for more detailed definitions); and $\bg = P_\sigma \f$. Since $A$ is a positive and self-adjoint operator with compact inverse, it admits an orthonormal basis of eigenvectors $\{\bw_i\}_{i\in \mathbb{N}}$. Then, for each $N \in \mathbb{N}$, we can consider the finite-dimensional space $H_N = \text{span}\{\bw_1,\ldots,\bw_N\} = P_N H$, with $P_N$ denoting the orthogonal projection of $H$ onto $H_N$. A numerical approximation of the solution $\bv$ of \eqref{DataAssAlgFunc} can be obtained by computing the Galerkin approximation $\bv_N \in P_N H$, which satisfies the following system of ordinary differential equations
\be\label{DataAssAlgFuncGalerkin}
\frac{\rd \bv_N}{\rd t} + \nu A \bv_N + P_N B(\bv_N,\bv_N) = P_N \bg - \beta P_N P_{\sigma}(I_h(\bv_N) - I_h(\bu)).
\ee

Notice that, since $\bv_N \in P_N H$, the error committed in approximating $\bv$ by $\bv_N$ must be greater than or equal to the error associated with the best approximation of $\bv$ in $P_N H$, $P_N \bv$, i.e.
\[
 |\bv - \bv_N|_{L^2} \geq |\bv - P_N \bv|_{L^2} = |Q_N \bv|_{L^2},
\]
where $Q_N = I - P_N$.

The Postprocessing Galerkin method provides us with an efficient way of obtaining a better approximation of $\bv$ than $\bv_N$. The idea consists in complementing the finite-dimensional approximation $\bv_N \in P_N H$ of $\bv$ with a suitable part lying in the complement space $Q_N H$. Adapting the algorithm introduced in \cite{GarciaNovoTiti1998,GarciaNovoTiti1999} to our situation, we can summarize it in the following steps:

For obtaining an approximation of $\bv$ at a certain time $T > t_0$,
\begin{enumerate}[(i)]
 \item\label{PPGMi} Integrate \eqref{DataAssAlgFuncGalerkin} in time, over the time interval $[t_0,T]$, to obtain $\bv_N$ and compute $\bv_N(T)$;
 \item\label{PPGMii} Obtain $\bq_N$ satisfying $\nu A \bq_N = Q_N[\f - B(\bv_N(T),\bv_N(T))]$;
 \item\label{PPGMiii} Compute the new approximation to $\bv(T)$, and hence to $\bu(T)$, given by $\bv_N(T) + \bq_N$.
\end{enumerate}


The equation satisfied by $\bq_N$ in step \eqref{PPGMii} is inspired by the definition of the approximate inertial manifold introduced in \cite{FoiasManleyTemam1988}, in which the authors obtain an approximation of $Q_N \bu$, with $\bu$ being a solution of \eqref{NSEeqs}, given by
\be\label{defPhi1}
 Q_N \bu \approx \Phi_1(P_N \bu) = (\nu A)^{-1} Q_N[\f - B(P_N \bu,P_N \bu)].
\ee
The graph of the mapping $\Phi_1: P_N H \to Q_N H$ is called an approximate inertial manifold. This approximation is obtained by applying the projection $Q_N$ to equation \eqref{NSEeqs} and, based on theoretical arguments, neglecting all lower order terms, i.e., the time derivative of $Q_N \bu$ and the nonlinear terms involving $Q_N \bu$, in comparison to the remaining terms. Since our idea is to ultimately obtain an approximation of $\bu$, it is natural to consider as an approximation of $Q_N \bv$ the same type of approximation used for $Q_N \bu$ in \eqref{defPhi1}, in which $P_N \bu$ is replaced by $\bv_N$, given that this is the approximation of $P_N \bu$ that we consider.

Our results show that the Postprocessing Galerkin method yields a better convergence rate than the standard Galerkin method, as also obtained in \cite{GarciaNovoTiti1999}.
However, an important difference in our results is that our error estimate is uniform in time, while in \cite{GarciaNovoTiti1998,GarciaNovoTiti1999} it grows exponentially in time. This remarkable difference is due to the fact that the approximate system \eqref{DataAssAlgFuncGalerkin} has a stabilizing mechanism imposed by the feedback control term, which kills the instabilities in the large (coarse) spatial scales caused by the nonlinear term. As a consequence, as proved in \cite[Theorems 1 and 2]{AzouaniOlsonTiti2014}, under suitable conditions on the parameters $\beta$ and $h$, the solutions of \eqref{DataAssAlg}, corresponding to arbitrary initial data, all converge to the same reference solution $\bu$. This shows that, with the appropriate conditions on $\beta$ and $h$, system \eqref{DataAssAlg} is globally asymptotically stable. Hence, the Galerkin approximation $\bv_N$ of $\bv$ converges to $\bv$ uniformly in time, as $N$ tends to infinity.

This stabilizing effect was also observed by the numerical computations performed in \cite{GeshoOlsonTiti2015} by using the Galerkin method (see also \cite{AltafTitiKnioZhaoMcCabeHoteit2015}), which showed that the required conditions on the parameters $\beta$ and $h$ are remarkably less strict than suggested by the analytical results in \cite{AzouaniOlsonTiti2014}. Consequently, this work provides a rigorous analytical justification for the computational study in \cite{AltafTitiKnioZhaoMcCabeHoteit2015,GeshoOlsonTiti2015}.

It is worth mentioning that the introduction of the Postprocessing Galerkin method was preceded by another spectral method also derived from the standard Galerkin approach and inspired by the idea of approximate inertial manifold, known as the Nonlinear Galerkin method (see, e.g., \cite{DevulderMarionTiti1993,FoiasJollyKevrekidisSellTiti1988, JollyKevrekidisTiti1990, MarionTemam1989} and references therein). The main difference between the two approaches is that, in the Postprocessing Galerkin method, the integration of the low modes does not use the information about the high modes: only at the final step the high modes are used in order to refine the solution (see step \eqref{PPGMii}, above). On the other hand, in the Nonlinear Galerkin method, the time step integration of the low modes is continuously updated by using the information on the high modes. For this reason, the Nonlinear Galerkin method, although providing a better error estimate in comparison to the standard Galerkin method, has the
disadvantage of being a lot more computationally expensive and thus, in
practice, less efficient (cf. \cite{GarciaNovoTiti1998,GarciaNovoTiti1999,GrahamSteenTiti1993,MargolinTitiWynne2003}). In an attempt of obtaining an algorithm that would overcome this disadvantage, while still keeping the better accuracy of the Nonlinear Galerkin method, the Postprocessing Galerkin method was developed.

Moreover, it is shown in \cite{MargolinTitiWynne2003}, by using a truncation analysis argument, that the Postprocessing Galerkin method is more than a technique for improving efficiency. The authors show that the Postprocessing Galerkin method is actually the correct leading order approximating scheme, and not the standard Galerkin method, as it is commonly believed.

We emphasize that, the case of continuous in time measurements and continuous in time Galerkin approximations were considered here for simplicity and in order to fix ideas. However, combining ideas from this work with those in \cite{FoiasMondainiTiti2016}, one can extend our results to the case of discrete in time measurements with errors and to discretize, accordingly, the Galerkin scheme \eqref{DataAssAlgFuncGalerkin} in time. This is a subject of future work. Moreover, we considered here the Galerkin approximation based on the eigenfunctions of the Stokes operator. However, following the ideas from \cite{GarciaArchillaTiti2000}, one can also employ the Postprocessing Galerkin method in the context of finite elements and apply it to our data assimilation scheme. Furthermore, since our data assimilation algorithm is inspired by feedback control ideas, we expect that our results will equally apply to feedback control systems.

This paper is organized as follows. In section \ref{secPreliminaries}, we provide a summary of the necessary background related to the two-dimensional Navier-Stokes equations that will be needed in the sequel. Section \ref{secMainResults} contains the main results of this paper. The purpose is to show a uniform in time estimate of the error committed when applying the Postprocessing Galerkin method described in \eqref{PPGMi}-\eqref{PPGMiii}, above, to system \eqref{DataAssAlg}, in order to obtain an approximation of the reference solution $\bu$ satisfying \eqref{NSEeqs} (Theorems \ref{thmerrorPPGM} and \ref{thmerrorPPGMIh}). We divide the presentation into two subsections: subsection \ref{subsecFouriermodescase} deals with the case of an interpolant operator given by a low Fourier modes projector; while subsection \ref{subsecAnotherClassIntOp} deals with a more general class of interpolant operators satisfying suitable properties, for which an example is given by the operator defined as local averages over
finite volume elements, in the case of periodic boundary conditions. Finally, in the Appendix, we show
for completeness that such example of interpolant operator verifies the properties considered in subsection \ref{subsecAnotherClassIntOp}.

\section{Preliminaries}\label{secPreliminaries}

In this section, we briefly recall the necessary background on the two-dimensional incompressible Navier-Stokes equations \eqref{NSEeqs}. For further details, see, e.g., \cite{bookcf1988,FMRT2001,Temambook1995,Temambook2001}.

Consider a spatial domain $\Omega \subset \mathbb{R}^2$ and a time interval $[t_0, \infty) \subset \mathbb{R}$. We assume, for simplicity, that the forcing $\f$ is time-independent and lies in the space $L^2(\Omega)^2$. We remark, however, that similar results are also valid in the case $\f \in L^{\infty}([t_0,\infty);L^2(\Omega)^2)$.

We consider two types of boundary conditions for system \eqref{NSEeqs}: periodic or no-slip Dirichlet. In the periodic case, we consider the fundamental domain $\Omega = (0,L) \times (0,L)$. Moreover, we assume that the velocity field and the pressure are periodic with period $L$ in each spatial direction $x_i$, $i=1,2$, and that $\f$ has zero spatial average, i.e.,
\[
 \int_{\Omega} \f(\bx) \rd \bx = 0.
\]
In the no-slip Dirichlet case, we consider $\Omega$ as a bounded subset of $\mathbb{R}^2$ with sufficiently smooth boundary $\partial \Omega$ and assume that $\bu = 0$ on $\partial \Omega$.

The definition of the space of test functions, denoted here by $\mV$, depends on the type of boundary condition being considered. In the periodic case, $\mV$ is defined as the set of all $L$-periodic trigonometric polynomials from $\mathbb{R}^2$ to $\mathbb{R}^2$ that are divergence free and have zero spatial average. In the no-slip Dirichlet case, we define $\mV$ as the family of $C^\infty$ vector fields with values in $\mathbb{R}^2$ that are divergence free and compactly supported in $\Omega$.

We denote by $H$ the closure of $\mV$ with respect to the norm in $L^2(\Omega)^2$, and by $V$ the closure of $\mV$ under the $H^1(\Omega)^2$ Sobolev norm. Following the notation from \cite{FMRT2001}, we denote the inner products in $H$ and $V$ by $( \cdot,\cdot)_{L^2} $ and $(\!( \cdot, \cdot )\!)_{H^1}$, respectively. They are defined as
\[
 (\bu,\bv)_{L^2} = \int_{\Omega} \bu(\bx) \cdot \bv(\bx) \rd \bx, \quad \forall \bu,\bv \in H,
\]
\[
 (\!(\bu,\bv)\!)_{H^1} = \int_{\Omega} \sum_{i=1}^2  \frac{\partial \bu}{\partial \bx_i} \cdot \frac{\partial \bv}{\partial x_i} \rd \bx, \quad \forall \bu,\bv \in V,
\]
and the associated norms are given by $|\bu|_{L^2} = (\bu,\bu)^{1/2}_{L^2}$, $\|\bu\|_{H_1} = (\!(\bu,\bu)\!)^{1/2}_{H^1}$.

The fact that $\|\cdot\|_{H^1}$ defines a norm in $V$ is justified via the Poincar\'e inequality, given by
\be\label{ineqPoincare}
 \lambda_1^{1/2}|\bu|_{L^2} \leq \|\bu\|_{H^1}, \quad \forall \bu \in V,
\ee
where $\lambda_1$ is the first eigenvalue of the Stokes operator, defined in \eqref{defStokesop}, below.

Given $R > 0$, we denote by $B_H(R)$ and $B_V(R)$ the closed balls centered at $0$ with radius $R$, with respect to the norms in $H$ and $V$, respectively.

We also consider the dual spaces of $H$ and $V$, denoted by $H'$ and $V'$, respectively. After identifying $H$ with its dual, we obtain $V \subseteq H  \subseteq V'$, with the injections being continuous, compact, and each space dense in the following one. Moreover, we denote the duality product between $V$ and $V'$ by $\langle \cdot, \cdot \rangle_{V',V}$.

Let $P_{\sigma}$ be the Leray-Helmholtz projector, i.e., the orthogonal projection of $L^2(\Omega)^2$ onto $H$. Applying $P_{\sigma}$ to system \eqref{NSEeqs}, we obtain its following equivalent functional formulation:
\be\label{functionaleq}
 \frac{\rd \bu}{\rd t} + \nu A\bu + B(\bu,\bu) = \bg \quad \mbox{in $V'$},
\ee
where $\bg = P_{\sigma} \f \in H$, $B:V\times V \rightarrow V'$ is the bilinear operator defined as the continuous extension of the operator given by
\[
 B(\bu,\bv) = P_{\sigma}((\bu \cdot \nabla)\bv), \quad \forall \bu, \bv \in \mV,
\]
and $A: D(A) \subseteq V \rightarrow V'$ is the Stokes operator, defined as the continuous extension of
\be\label{defStokesop}
 A\bu = -P_{\sigma} \Delta \bu, \quad \forall \bu \in \mV,
\ee
with the domain of $A$, $D(A)$, given by $V\cap H^2(\Omega)^2$.

The Stokes operator is a positive and self-adjoint operator with compact inverse. Therefore, it admits an orthonormal basis of eigenvectors $\{\bw_m\}_{m\in \mathbb{N}}$ associated with a nondecreasing sequence of positive eigenvalues $\{\lambda_m\}_{m\in\mathbb{N}}$, with $\lambda_m \to \infty$ as $m \to \infty$.

We also consider, for each $N \in \mathbb{N}$, the low modes projector $P_N$, which is defined as the orthogonal projector of $H$ onto the subspace $H_N = \textnormal{span}\{ \bw_1,\ldots,\bw_N \}$. Moreover, we denote $Q_N = I - P_N$.

The bilinear operator $B$ satisfies the following property:
\be\label{PropBilTerm}
 \langle B(\bu_1,\bu_2), \bu_3\rangle_{V',V} = -\langle B(\bu_1,\bu_3), \bu_2 \rangle_{V',V}, \quad \forall \bu_1,\bu_2, \bu_3 \in V.
\ee


Recall the Br\'ezis-Gallouet inequality \cite{BrezisGallouet1980,FoiasManleyTemamTreve1983}, given by
\be\label{ineqBrezisGallouet}
\|\bu\|_{L^\infty} \leq c_B \|\bu\|_{H^1} \left[ 1 + \log\left( \frac{|A\bu|_{L^2}}{\lambda_1^{1/2} \|\bu\|_{H^1}} \right) \right]^{1/2}, \quad \forall \bu \in D(A),
\ee
where $c_B$ is a nondimensional (scale invariant) constant, and $\|\cdot\|_{L^\infty}$ denotes the usual norm in $L^\infty(\Omega)^2$.

We now recall some inequalities satisfied by the bilinear term $B$. Using Br\'ezis-Gallouet inequality \eqref{ineqBrezisGallouet}, we obtain that, for every $\bu_1 \in D(A)$ with $\bu_1 \neq 0$ and every $\bu_2 \in V$ and $\bu_3 \in H$,
\be\label{ineqConseqBrezisGallouet}
|(B(\bu_1,\bu_2),\bu_3)_{L^2}| \leq c_B \|\bu_1\|_{H^1} \|\bu_2\|_{H^1} |\bu_3|_{L^2} \left[ 1 + \log\left( \frac{|A \bu_1|_{L^2}}{\lambda_1^{1/2}\|\bu_1\|_{H^1}}\right) \right]^{1/2}.
\ee

We also recall the following logarithmic inequalities from \cite{Titi1987}:

For every $\bu_1, \bu_2, \bu_3 \in V$, with $\bu_3 \neq 0$,
\be\label{ineqTiti1}
|(B(\bu_1,\bu_2),\bu_3)_{L^2}| \leq c_T \|\bu_1\|_{H^1} \|\bu_2\|_{H^1} |\bu_3|_{L^2} \left[ 1 + \log\left( \frac{\|\bu_3\|_{H^1}}{\lambda_1^{1/2}|\bu_3|_{L^2}}\right) \right]^{1/2}.
\ee

For every $\bu_1 \in V$, and every $\bu_2, \bu_3 \in D(A)$, with $\bu_2 \neq 0$,
\be\label{ineqTiti2}
|(B(\bu_1,\bu_2),A\bu_3)_{L^2}| \leq c_T \|\bu_1\|_{H^1} \|\bu_2\|_{H^1} |A \bu_3|_{L^2} \left[ 1 + \log\left( \frac{|A \bu_2|_{L^2}}{\lambda_1^{1/2}\|\bu_2\|_{H^1}}\right) \right]^{1/2}.
\ee

Also, for every $\bu_1,\bu_2 \in V$, we have
\[
 B(\bu_1,\bu_1) - B(\bu_2,\bu_2) = B\left( \bu_1 - \bu_2, \frac{\bu_1 + \bu_2}{2} \right) + B\left( \frac{\bu_1 + \bu_2}{2}, \bu_1 - \bu_2 \right).
\]

Then, it follows from the result in \cite[Proposition 6.1]{bookcf1988} that, for every $\alpha > 1/2$,
\be\label{ineqA-alphadiffBuBv}
 |A^{-\alpha}(B(\bu,\bu) - B(\bv,\bv))| \leq c_\alpha  |\Omega|^{\alpha - \frac{1}{2}} \|\bu + \bv\|_{H^1} |\bu - \bv|_{L^2},
\ee
where $|\Omega|$ denotes the area of $\Omega$ and $c_\alpha$ is a constant depending on $\alpha$ through the Sobolev constants from the Sobolev embeddings of $H^{2\alpha}(\mathbb{R}^2)$ into $L^{\infty}(\mathbb{R}^2)$, and of $H^s(\mathbb{R}^2)$ into $L^q(\mathbb{R}^2)$, with $1 > s > (2 - 2 \alpha)$, and $q = 2/(1-s)$. Thus, $c_\alpha \to \infty$ as $\alpha \to \frac{1}{2}^+$.

Along this paper, we denote by $c$ a positive absolute constant or a nondimensional positive constant depending on $\Omega$, whose value may change from line to line; while the capital letter $C$ denotes a dimensional constant, depending on the physical parameters, such as $\nu$, $\lambda_1$ and $|\bg|_{L^2}$.

Finally, we recall some results concerning uniform bounds, with respect to the norms in $H$ and $V$, for the solutions of  \eqref{NSEeqs}. It is well-known that, given $\bu_0 \in H$, there exists a unique weak solution of \eqref{NSEeqs} satisfying $\bu(t_0) = \bu_0$ and $\bu \in \mC([t_0, \infty); H) \cap L^2_{\text{loc}}(t_0, \infty; V)$, with $\rd \bu/\rd t \in L^2_{\text{loc}}(t_0, \infty; V')$. From now on, whenever we refer to a solution of \eqref{NSEeqs}, we mean a solution in this sense.

The proof of the next proposition can be found in any of the references listed above (\cite{bookcf1988,FMRT2001,Temambook1995,Temambook2001}). Recall the definition of the Grashof number, which is the nondimensional quantity given by
\[  G = \frac{|\bg|_{L^2}}{\nu^2\lambda_1}.
\]


\begin{prop}\label{propunifbounds}
 Let $\bu_0 \in H$ and let $\bu$ be a solution of \eqref{NSEeqs} satisfying $\bu(t_0) = \bu_0$. Then, there exists $T = T(\nu, \lambda_1, |\bg|_{L^2},|\bu_0|_{L^2}) \geq t_0$ such that the following hold:
 \begin{enumerate}[(i)]
  \item In the case of periodic boundary conditions,
  \be\label{unifboundperiodic} |\bu(t)|_{L^2} \leq 2\nu G, \quad \|\bu(t)\|_{H^1} \leq 2\nu \lambda_1^{1/2} G, \quad \forall t \geq T.
  \ee
  \item In the case of no-slip boundary conditions,
  \be\label{unifboundnoslip} |\bu(t)|_{L^2} \leq 2\nu G, \quad \|\bu(t)\|_{H^1} \leq c \nu \lambda_1^{1/2} G \Exp^{\frac{G^4}{2}}, \quad \forall t \geq T.
  \ee
 \end{enumerate}
\end{prop}

In order to simplify the notation, we write the uniform bounds in the $H$ and $V$ norms from Proposition \ref{propunifbounds} by using constants $M_0$ and $M_1$, respectively, i.e.,
\be\label{unifboundsattractor}
|\bu(t)|_{L^2} \leq M_0, \quad \|\bu(t)\|_{H^1} \leq M_1, \quad \forall t \geq T.
\ee
Notice that the value of $M_1$ changes according to the boundary condition being considered.

The following theorem follows immediately from the result proved in \cite[Theorem 1.1]{FoiasManleyTemam1988} (see also \cite{Titi1990}).

\begin{thm}\label{unifboundsQm}
 Let $\bu_0 \in H$ and let $\bu$ be a solution of \eqref{NSEeqs} satisfying $\bu(t_0) = \bu_0$. Then, there exists $T = T(\nu, \lambda_1,|\bg|_{L^2},|\bu_0|_{L^2}) \geq t_0$ such that
 \be\label{boundqL2}
|Q_N \bu (t)|_{L^2} \leq C_0 \frac{L_N}{\lambda_{N+1}}, \quad \forall t \geq T, \quad \forall N \in \mathbb{N},
\ee
\be\label{boundqH1}
\|Q_N \bu (t)\|_{H^1} \leq C_1 \frac{L_N}{\lambda_{N+1}^{1/2}}, \quad \forall t \geq T, \quad \forall N \in \mathbb{N},
\ee
where
\be\label{defLN}
L_N = \left[ 1 + \log\left( \frac{\lambda_N}{\lambda_1}\right) \right]^{1/2},
\ee
\be\label{defC0}
C_0 = c\left( \frac{|Q_N \bg|_{L^2} + M_1^2}{\nu} \right),
\ee
\be\label{defC1}
C_1 = c \left( \frac{|Q_N \bg|_{L^2} + M_1^2}{\nu} + \frac{M_0 M_1^2}{\nu^2} \right),
\ee
and $M_0$ and $M_1$ are as given in \eqref{unifboundsattractor}.
\end{thm}

The next theorem was proved in \cite[Theorem 2.1]{FoiasManleyTemam1988}, and it provides uniform in time estimates, in the $H$ and $V$ norms, for the distance between a solution $\bu$ of \eqref{NSEeqs} and its vertical projection on the graph of the mapping $\Phi_1$, given in \eqref{defPhi1}.

\begin{thm}\label{thmdistuPhi1}
 Let $\bu_0 \in H$ and let $\bu$ be a solution of \eqref{NSEeqs} satisfying $\bu(t_0) = \bu_0$. Then, there exists $T = T(\nu, \lambda_1,|\bg|_{L^2},|\bu_0|_{L^2}) \geq t_0$ such that
 \be\label{distuPhi1L2aa}
 |\Phi_1(P_N \bu(t)) - Q_N \bu(t)|_{L^2} \leq C \frac{L_N}{\lambda_{N+1}^{3/2}}, \quad \forall t \geq T, \quad \forall N \in \mathbb{N},
\ee
and
\be\label{distuPhi1H1aa}
 \|\Phi_1(P_N\bu(t)) - Q_N \bu(t)\|_{H^1} \leq C \frac{L_N}{\lambda_{N+1}}, \quad \forall t \geq T, \quad \forall N \in \mathbb{N},
\ee
where $C$ is a constant depending on $\nu$, $\lambda_1$ and $|\bg|_{L^2}$, but independent of $N$.
\end{thm}

\begin{rmk}\label{rmkunifboundsu}
In the results of section \ref{secMainResults}, we will assume that the reference solution of \eqref{NSEeqs} has evolved long enough so that the uniform bounds from Proposition \ref{propunifbounds}, Theorem \ref{unifboundsQm} and Theorem \ref{thmdistuPhi1} are always valid, i.e., for simplicity, we assume that $T = t_0$. Notice that, in particular, the uniform bounds from Proposition \ref{propunifbounds}, Theorem \ref{unifboundsQm} and Theorem \ref{thmdistuPhi1} are valid for any trajectory $\bu = \bu(t)$ lying in the global attractor of \eqref{NSEeqs}, for every $t \in \mathbb{R}$.
\end{rmk}

\section{Main Results}\label{secMainResults}

The purpose of this section is to establish analytical estimates of the error that occurs when using the Postprocessing Galerkin method applied to the data assimilation algorithm \eqref{DataAssAlg} in order to obtain an approximation of the reference solution $\bu$, which satisfies the 2D Navier-Stokes equations \eqref{NSEeqs}. This means we want to establish an estimate of the difference $[(\bv_N + \Phi_1(\bv_N)) - \bu]$ in some appropriate norm, where $\bv_N$ denotes the Galerkin approximation of $\bv$, the solution of \eqref{DataAssAlg}, in $P_N H$. This is done here for the norms in the spaces $H$ and $V$.

We start by giving some of the main ideas behind our results. From now on, we reserve the letter $N \in \mathbb{N}$ for the number of modes in the Galerkin approximation of \eqref{DataAssAlg}, and we adopt the following notation for the low and high modes of the reference solution $\bu$: $\bp = P_N \bu$ and $\bq = Q_N \bu$, respectively. Moreover, we assume that $\bu$ satisfies the bounds from \eqref{unifboundsattractor}-\eqref{distuPhi1H1aa}, for every $t \geq t_0$.

First, we rewrite the error in implementing the Postprocessing Galerkin method as
\be\label{rewritingerror}
(\bv_N + \Phi_1(\bv_N)) - \bu = (\bv_N - \bp) + (\Phi_1(\bp) - \bq) + (\Phi_1(\bv_N) - \Phi_1(\bp)).
\ee

From Theorem \ref{thmdistuPhi1}, we have that, for every $t \geq t_0$,
\be\label{distuPhi1L2a}
 |\Phi_1(\bp(t)) - \bq(t)|_{L^2} \leq C \frac{L_N}{\lambda_{N+1}^{3/2}}
\ee
and
\be\label{distuPhi1H1a}
 \|\Phi_1(\bp(t)) - \bq(t)\|_{H^1} \leq C \frac{L_N}{\lambda_{N+1}},
\ee
where $C = C(\nu, \lambda_1, |\bg|_{L^2})$.

Moreover, it is not difficult to see that the restriction of $\Phi_1$ to the set $P_N B_V(R)$, for any $R > 0$, is a Lipschitz continuous mapping with respect to the norms in both $H$ and $V$ (see, e.g., \cite[Appendix]{DevulderMarionTiti1993}). More specifically, we have
\be\label{Phi1LipschitzL2}
 |\Phi_1(\bp_1) - \Phi_1(\bp_2)|_{L^2} \leq l |\bp_1 - \bp_2|_{L^2}, \quad \forall \bp_1, \bp_2 \in P_N B_V(R),
\ee
and
\be\label{Phi1LipschitzH1}
 \|\Phi_1(\bp_1) - \Phi_1(\bp_2)\|_{H^1} \leq l \|\bp_1 - \bp_2\|_{H^1}, \quad \forall \bp_1, \bp_2 \in P_N B_V(R),
\ee
where $l = C\lambda_{N+1}^{-1/4}$, with $C$ being a constant depending on $\nu$, $\lambda_1$ and $R$.

It follows from Propositions \ref{thmUnifBoundvNH1Fourier} and \ref{thmUnifBoundvNH1} below that, given a solution $\bu$ of \eqref{NSEeqs} satisfying \eqref{unifboundsattractor}-\eqref{boundqH1}, for every $t \geq t_0$, and given $\bv_0 \in B_V(M_1)$, under suitable conditions on the parameters $\beta$ and $h$, the solution $\bv_N$ of \eqref{DataAssAlg}, with $\bv_N(t_0) = P_N \bv_0$, satisfies $\bv_N(t) \in B_V(3M_1)$, for all $t \geq t_0$. Thus, using \eqref{distuPhi1L2a}, \eqref{distuPhi1H1a} and \eqref{Phi1LipschitzL2}-\eqref{Phi1LipschitzH1} with $R = 3M_1$, we obtain from \eqref{rewritingerror} that, for every $t \geq t_0$,
\be\label{distuPhi1L2b}
|(\bv_N(t) + \Phi_1(\bv_N(t))) - \bu(t)|_{L^2} \leq (1+l)|\bv_N(t) - \bp(t)|_{L^2} + C \frac{L_N}{\lambda_{N+1}^{3/2}},
\ee
and
\be\label{distuPhi1H1b}
\|(\bv_N(t) + \Phi_1(\bv_N(t))) - \bu(t)\|_{H^1} \leq (1+l)\|\bv_N(t) - \bp(t) \|_{H^1} + C \frac{L_N}{\lambda_{N+1}}.
\ee

Moreover, we also have
\be\label{inversePoincarevNPNu}
\|\bv_N(t) - \bp(t)\|_{H^1} \leq \lambda_N^{1/2} |\bv_N(t) - \bp(t)|_{L^2}.
\ee

Thus, using also \eqref{inversePoincarevNPNu}, we see from \eqref{distuPhi1L2b} and \eqref{distuPhi1H1b} that it suffices to obtain an estimate for $|\bv_N(t) - \bp(t)|_{L^2}$ in order to achieve our goal.

Applying $P_N$ to \eqref{functionaleq}, we see that $\bp = P_N \bu$ satisfies the equation
\be\label{NSEfuncPNintrosecresults}
 \frac{\rd \bp}{\rd t} + \nu A \bp + P_N B(\bp, \bp) - P_N \bg = - P_N G,
\ee
where
\begin{multline}
 G(t) = B(\bu(t),\bu(t)) - B(\bp(t),\bp(t)) \\
      = B(\bp(t),\bq(t)) + B(\bq(t),\bp(t)) + B(\bq(t),\bq(t));
\end{multline}
while we recall from \eqref{DataAssAlgFuncGalerkin} that $\bv_N$ satisfies the equation
\be\label{DataAssAlgFuncGalerkinintrosecresults}
 \frac{\rd \bv_N}{\rd t} + \nu A \bv_N + P_N B(\bv_N,\bv_N) - P_N \bg = - \beta P_N P_{\sigma}I_h(\bv_N - \bu).
\ee
Now, denoting $\bw = \bv_N - \bp$ and taking the difference between \eqref{NSEfuncPNintrosecresults} and \eqref{DataAssAlgFuncGalerkinintrosecresults}, we obtain that
\begin{multline}\label{eqwintrosecresults}
 \frac{\rd \bw}{\rd t} + \nu A \bw + \beta \bw + P_N [B(\bv_N,\bv_N) - B(\bp,\bp)] = P_N G - \beta P_N P_\sigma [I_h(\bw) - \bw] \\
 + \beta P_N P_\sigma I_h (\bq).
\end{multline}
The terms $\nu A \bw$ and $\beta\bw$ represent the dissipative terms in \eqref{eqwintrosecresults}, which act on stabilizing $\bw$. The term $A\bw$ has a stronger effect than $\beta \bw$ on the high modes of $\bw$, for small values of $\nu$; while $\beta\bw$ has a stronger effect than $\nu A \bw$ on the low modes of $\bw$.

Applying the Duhamel's (variation of constants) formula to \eqref{eqwintrosecresults}, yields, for every $s \geq t$,
\begin{multline}\label{eqwDuhamelintrosecresults}
 \bw(s) = \Exp^{-(s-t) (\nu A P_N + \beta P_N)} \bw(t) \\
  - \int_{t}^s \Exp^{-(s-\tau) (\nu A P_N + \beta P_N)} P_N [ B(\bv_N(\tau),\bv_N(\tau)) - B(\bp(\tau),\bp(\tau))] \rd \tau \\
  + \int_{t}^s \Exp^{-(s-\tau) (\nu A P_N + \beta P_N)} P_N G(\tau) \rd \tau \\
  - \beta \int_{t}^s \Exp^{-(s-\tau) (\nu A P_N + \beta P_N)} P_N [P_\sigma I_h(\bw(\tau)) - \bw(\tau)] \rd \tau \\
  + \beta \int_{t}^s \Exp^{-(s-\tau) (\nu A P_N + \beta P_N)}  P_N P_\sigma I_h(\bq(\tau)) \rd \tau .
\end{multline}

The estimates for the terms on the right-hand side of \eqref{eqwDuhamelintrosecresults} are obtained by taking advantage of the smoothing effect of the operator $\Exp^{-(s-t) (\nu A P_N + \beta P_N)}$, with the finite-dimensionality of the operator $P_N$ also playing a crucial role. Moreover, the estimates for the last two terms on the right-hand side of \eqref{eqwDuhamelintrosecresults} are obtained by using suitable properties of the interpolant operator $I_h$.

We consider two types of interpolant operators $I_h$, treated in two different sections. In the first one, section \ref{subsecFouriermodescase}, we consider $I_h$ as a low Fourier modes projector, i.e., $I_h = P_K$, $K \in \mathbb{N}$. In this case, we notice that we can commute $P_N$ with $I_h = P_K$ and thus the last term on the right-hand side of \eqref{eqwintrosecresults} is zero, which simplifies the analysis.

In section \ref{subsecAnotherClassIntOp}, we consider a more general class of interpolation operators, satisfying suitable properties (see properties \eqref{propintP1}-\eqref{propintP3} in section \ref{subsecAnotherClassIntOp}, below), which are, in particular, satisfied by the example of a low Fourier modes projector considered in section \ref{subsecFouriermodescase}. Another particular example of such class of interpolant operators is given by local averages over finite volume elements, which is illustrated in the Appendix in the case of periodic boundary conditions. In this latter example, this approach can be viewed as a hybrid method, in the sense that observations are acquired through a finite elements method, while the approximate model is numerically solved through a spectral method, the Postprocessing Galerkin.

The proof of the estimate for $|\bv_N - \bp|_{L^2}$, in both cases, follows similar ideas to the proof given in \cite[Theorem 2]{GarciaNovoTiti1999}, where, for a given initial condition $\bu(t_0) = \bu_0$, an estimate was obtained for $|\bu_N - \bp|_{L^2}$, with $\bu_N$ being the Galerkin approximation of $\bu$ satisfying $P_N \bu_N(t_0) = P_N \bu_0$. We remark, however, that an advantage of our result is that the estimate for $|\bv_N - \bp|_{L^2}$ is uniform in time (see Theorems \ref{thmestvNPNu} and \ref{thmestvNPNuIh}), while the estimate for $|\bu_N - \bp|_{L^2}$ given in \cite[Theorem 2]{GarciaNovoTiti1999} grows exponentially in time. This important difference is justified by the presence of the additional dissipative term $\beta \bw$ in \eqref{eqwintrosecresults}, which helps to stabilize the large scales of $\bw$ when the parameter $\beta$ is suitably chosen. More specifically, $\beta$ needs to be chosen large enough in order to stabilize the large
spatial scales of $\bw$, but not too large so as not to destabilize the small spatial scales of $\bw$ as well,
which are dissipated by $\nu A \bw$, for small values of $\nu$. For this reason, we need, roughly, $\beta \leq c\nu/h^2$.

Using the previous ideas, we prove in Theorems \ref{thmestvNPNu} and \ref{thmestvNPNuIh} below that, for sufficiently large $t$,
\be\label{errorlowmodes}
|\bv_N(t) - \bp(t)|_{L^2} \leq O(L_N^4 \lambda_{N+1}^{-3/2}),
\ee
in the case of an interpolant operator given by a low Fourier modes projector; and
\be\label{errorlowmodesIh}
|\bv_N(t) - \bp(t)|_{L^2} \leq O(L_N \lambda_{N+1}^{-5/4}),
\ee
in the general interpolant operator case.

Thus, from \eqref{distuPhi1L2b} and \eqref{distuPhi1H1b}, it follows that, for $t$ large enough,
\be\label{ordererrorPPGML2}
|(\bv_N(t) + \Phi_1(\bv_N(t))) - \bu(t)|_{L^2} \leq O(L_N^4 \lambda_{N+1}^{-3/2})
\ee
and
\be\label{ordererrorPPGMH1}
\|(\bv_N(t) + \Phi_1(\bv_N(t))) - \bu(t)\|_{H^1} \leq O(L_N^4 \lambda_{N+1}^{-1}),
\ee
in the case of an interpolant operator given by a low Fourier modes projector (cf. Theorem \ref{thmerrorPPGM}); and
\be\label{ordererrorPPGML2Ih}
|(\bv_N(t) + \Phi_1(\bv_N(t))) - \bu(t)|_{L^2} \leq O(L_N \lambda_{N+1}^{-5/4})
\ee
and
\be\label{ordererrorPPGMH1Ih}
\|(\bv_N(t) + \Phi_1(\bv_N(t))) - \bu(t)\|_{H^1} \leq O(L_N \lambda_{N+1}^{-3/4}),
\ee
in the general interpolant operator case (cf. Theorem \ref{thmerrorPPGMIh}, below).

On the other hand, from  \eqref{boundqL2}, \eqref{boundqH1}, \eqref{inversePoincarevNPNu}, \eqref{errorlowmodes} and \eqref{errorlowmodesIh}, we obtain that the error between the Galerkin approximation $\bv_N$ of $\bv$ and the reference solution $\bu$ satisfies, for $t$ large enough,
\be\label{ordererrorSGML2}
|\bv_N(t) - \bu(t)|_{L^2} \leq |\bv_N(t) - \bp(t)|_{L^2} + |\bq(t)|_{L^2} \leq O(L_N \lambda_{N+1}^{-1}),
\ee
\be\label{ordererrorSGMH1}
\|\bv_N(t) - \bu(t)\|_{H^1} \leq \|\bv_N(t) - \bp(t)\|_{H^1} + \|\bq(t)\|_{H^1} \leq O(L_N \lambda_{N+1}^{-1/2}),
\ee
in both cases of interpolant operators (cf. Corollaries \ref{corerrorSGM} and \ref{corerrorSGMIh}, below).

Comparing \eqref{ordererrorPPGML2} and \eqref{ordererrorPPGML2Ih} with \eqref{ordererrorSGML2}, and \eqref{ordererrorPPGMH1} and \eqref{ordererrorPPGMH1Ih} with \eqref{ordererrorSGMH1}, we see that, as mentioned in the Introduction, the Postprocessing Galerkin method indeed yields a better convergence rate than the standard Galerkin method. Notably, this improved rate is achieved due to essentially three facts: firstly, by exploring the fact that the error in the low modes, $|\bv_N - \bp|_{L^2}$, is much smaller than the error committed in the high modes, $|\bq|_{L^2}$ (cf. \eqref{errorlowmodes}, \eqref{errorlowmodesIh} and \eqref{boundqL2}), when using the standard Galerkin method; secondly, by complementing the finite-dimensional approximation $\bv_N \in P_N H$ with a suitable approximation of the high modes, given by $\Phi_1(\bv_N) \in Q_N H$, which yields a better approximation to $\bq$ than $0$ (cf. \eqref{distuPhi1L2a}-\eqref{distuPhi1H1a} and \eqref{boundqL2}-\eqref{boundqH1}); and finally, by using
the Lipschitz property of $\Phi_1$ (cf. \eqref{Phi1LipschitzL2}, \eqref{Phi1LipschitzH1}).

\begin{rmk}
We notice that the convergence rates with respect to $N$ in \eqref{ordererrorPPGML2Ih}-\eqref{ordererrorPPGMH1Ih}, obtained for the error committed when implementing the Postprocessing Galerkin method to \eqref{DataAssAlg} in the general interpolant operator case, is not as good as the rate in \eqref{ordererrorPPGML2}-\eqref{ordererrorPPGMH1}, for the case of an interpolant operator given by a low Fourier modes projector. In general terms, as pointed out before, this is due to the fact that the former case concerns a hybrid method, where the observations are acquired through, e.g., a finite elements method, while the approximate model \eqref{DataAssAlg} is discretized in space through a spectral method, the Postprocessing Galerkin. On the mathematical side, this is represented by the possible lack of commutativity between the operators $P_\sigma I_h$ and $A$, an issue that does not occur in the case of an interpolant operator given by a low Fourier modes projector, and which introduces additional error to
the estimates.
\end{rmk}


\subsection{The case of an interpolant operator given by a low Fourier modes projector}\label{subsecFouriermodescase}

We consider an interpolant operator given by the orthogonal projection on low modes of the Fourier domain, i.e. $I_h = P_K$, for some $K\in \mathbb{N}$. The data assimilation algorithm \eqref{DataAssAlgFunc} is given in this particular case by

\be\label{eqDataAssFourier}
\frac{\rd \bv}{\rd t} + \nu A \bv + B(\bv,\bv) = \bg - \beta P_K(\bv - \bu).
\ee

For every $N \in \mathbb{N}$ with $N \geq K$, we consider the Galerkin approximation system of \eqref{eqDataAssFourier} in the space $P_N H$, given by
\begin{eqnarray}\label{eqDataAssFourierGalerkin}
\frac{\rd \bv_N}{\rd t} + \nu A \bv_N + P_N B(\bv_N,\bv_N) = P_N \bg - \beta P_K(\bv_N - \bu) \nonumber \\
= P_N \bg - \beta P_K(\bv_N - \bp),
\end{eqnarray}
with an initial condition given by
\be
\bv_N(t_0) = P_N \bv_0,
\ee
where $\bv_0$ is chosen in a suitable space, but arbitrarily. We assume either periodic or no-slip Dirichlet boundary conditions.

The condition $N \geq K$ is assumed here for simplicity purposes. Nevertheless, it is a natural assumption, since one would expect to have the resolution of the numerical method to be greater or equal than the resolution associated to the observations.

The following result provides a first uniform in time bound of the finite-dimensional difference $\bv_N - \bp$ in the $H^1$ norm, under suitable conditions on $\beta$ and $K$. Since we assume that the reference solution $\bu$ satisfies the bounds from \eqref{unifboundsattractor}-\eqref{boundqH1}, for every $t\geq t_0$, we also have in particular that $\bp$ is uniformly bounded in $V$. Thus, as a consequence of the following proposition, we obtain that $\bv_N$ is also uniformly bounded in $V$, provided $\beta$ and $K$ satisfy the appropriate conditions.

In the statement below, we consider an auxiliary parameter $m \in \mathbb{N}$ that is used for one of the lower bounds needed for $\beta$. More specifically, we choose $\beta$ such that, in particular, $\beta \geq \nu \lambda_m$. This auxiliary parameter plays a more important role in the proof of Theorem \ref{thmestvNPNu}, below, but we also use it here in order to be consistent.

\begin{prop}\label{thmUnifBoundvNH1Fourier}
 Let $\bu$ be a solution of \eqref{NSEeqs} satisfying \eqref{unifboundsattractor}-\eqref{boundqH1}, for every $t \geq t_0$. Let $\bv_0 \in B_V(M_1)$, with $M_1$ as in \eqref{unifboundsattractor}. For every $N \in \mathbb{N}$, let $\bv_N$ be the unique solution of \eqref{eqDataAssFourierGalerkin} satisfying $\bv_N(t_0) = P_N \bv_0$.
 Consider $m \in \mathbb{N}$ large enough such that
 \be\label{condmFourier}
  \lambda_m \geq \max \left\{ \frac{\lambda_1 \Exp}{2}, c \frac{C_1}{\nu}L_m^2, c \left( \frac{C_1^2}{\nu M_1}\right)^{2/3} L_m^2\right\}.
 \ee
 If $\beta > 0$ and $K \in \mathbb{N}$ are large enough such that
 \be\label{condbetaFourier}
  \beta \geq \max\left\{ \nu \lambda_m, c \frac{M_1^2}{\nu}\left[ 1 + \log \left(\frac{M_1}{\nu\lambda_1^{1/2}}\right)\right] \right\}
 \ee
 and
 \be\label{condKFourier}
  \lambda_{K+1} \geq \frac{2\beta}{\nu},
 \ee
 then, for every $N \geq K$,
 \be\label{unifboundvNFourier}
  \sup_{t \geq t_0} \|\bv_N(t) - \bp(t)\|_{H^1} \leq 2 M_1.
 \ee
\end{prop}
\begin{proof}
  Projecting \eqref{NSEeqs} onto $P_N H$, we have
  \be\label{eqNSEprojPN}
    \frac{\rd \bp}{\rd t} + \nu A \bp + P_N B(\bu, \bu) = P_N \bg.
  \ee

  Denote $\bw = \bv_N - \bp$. Subtracting \eqref{eqNSEprojPN} from \eqref{eqDataAssFourierGalerkin}, we obtain that
 \be\label{eqwFourier}
  \frac{\rd \bw}{\rd t} + \nu A \bw + P_N [ B(\bv_N,\bv_N) - B(\bu,\bu)] = -\beta P_K \bw.
 \ee

 Notice that
 \begin{multline}\label{rewritenonlinearterms}
  B(\bv_N,\bv_N) - B(\bu,\bu) =  B(\bv_N,\bv_N) - B(\bp + \bq, \bp + \bq) \\
  = B(\bv_N,\bv_N) - B(\bp,\bp) - B(\bp, \bq) - B(\bq,\bp) - B(\bq,\bq) \\
  = B(\bw, \bp) + B(\bp,\bw) + B(\bw,\bw) - B(\bp, \bq) - B(\bq,\bp) - B(\bq,\bq),
 \end{multline}

 Thus, from \eqref{eqwFourier} and \eqref{rewritenonlinearterms}, we have
 \begin{multline}\label{eqw2Fourier}
   \frac{\rd \bw}{\rd t} + \nu A \bw = -P_N[ B(\bw, \bp) + B(\bp,\bw) + B(\bw,\bw) - B(\bp, \bq) - B(\bq,\bp) - B(\bq,\bq)] \\
   -\beta P_K \bw.
 \end{multline}

 Taking the inner product in $L^2$ of \eqref{eqw2Fourier} with $A\bw$, yields
 \begin{multline}\label{enstrophyeqwFourier}
  \frac{1}{2} \frac{\rd}{\rd t} \|\bw\|_{H^1}^2 + \nu |A\bw|_{L^2}^2 = -(B(\bw, \bp),A\bw)_{L^2} - (B(\bp,\bw),A\bw)_{L^2} \\
  - (B(\bw,\bw),A\bw)_{L^2} + (B(\bp,\bq),A\bw)_{L^2} + (B(\bq,\bp),A\bw)_{L^2} + (B(\bq,\bq),A\bw)_{L^2} \\
  - \beta\|P_K \bw\|_{H^1}^2.
 \end{multline}

 Now we estimate the terms in the right-hand side of \eqref{enstrophyeqwFourier}.

 Using \eqref{ineqConseqBrezisGallouet} and \eqref{unifboundsattractor}, we obtain that
 \be\label{est1Fourier}
  |(B(\bw, \bp),A\bw)_{L^2}| \leq c_B M_1 \|\bw\|_{H^1} |A \bw|_{L^2} \left[ 1 + \log\left( \frac{|A \bw|_{L^2}}{\lambda_1^{1/2}\|\bw\|_{H^1}}\right) \right]^{1/2},
 \ee
 \be\label{est2Fourier}
  |(B(\bw, \bw),A\bw)_{L^2}| \leq c_B \|\bw\|_{H^1}^2 |A \bw|_{L^2} \left[ 1 + \log\left( \frac{|A \bw|_{L^2}}{\lambda_1^{1/2}\|\bw\|_{H^1}}\right) \right]^{1/2}.
 \ee

 Thanks to \eqref{ineqConseqBrezisGallouet}, \eqref{unifboundsattractor} and \eqref{boundqH1}, we have
 \begin{multline}\label{est3Fourier}
  |(B(\bp, \bq), A\bw)_{L^2}| \leq c_B \|\bp\|_{H^1} \|\bq\|_{H^1} |A\bw|_{L^2} \left[ 1 + \log\left( \frac{|A \bp|_{L^2}}{\lambda_1^{1/2}\|\bp\|_{H^1}}\right) \right]^{1/2} \\
  \leq c_B M_1 C_1 \frac{L_N^2}{\lambda_{N+1}^{1/2}} |A\bw|_{L^2} \leq \frac{\nu}{12}|A\bw|_{L^2}^2 + c \frac{C_1^2}{\nu} \frac{L_N^4}{\lambda_{N+1}}M_1^2.
 \end{multline}

 From \eqref{ineqTiti1} and \eqref{boundqH1}, it follows that
 \begin{multline}\label{est4Fourier}
  |(B(\bq, \bq), A\bw)_{L^2}| \leq c_T \|\bq\|_{H^1}^2 |A\bw|_{L^2} \left[ 1 + \log\left( \frac{|A^{3/2} \bw|_{L^2}}{\lambda_1^{1/2}|A\bw|_{L^2}}\right) \right]^{1/2}\\
  \leq c_T C_1^2  \frac{L_N^3}{\lambda_{N+1}} |A\bw|_{L^2} \leq \frac{\nu}{12} |A\bw|_{L^2}^2 + c \frac{C_1^4}{\nu} \frac{L_N^6}{\lambda_{N+1}^2}.
 \end{multline}

 From \eqref{ineqTiti2} and \eqref{unifboundsattractor}, we obtain that
 \be\label{est5Fourier}
  |(B(\bp,\bw),A\bw)_{L^2}| \leq c_T M_1 \|\bw\|_{H^1} |A \bw|_{L^2}\left[ 1 + \log\left( \frac{|A \bw|_{L^2}}{\lambda_1^{1/2}\|\bw\|_{H^1}}\right) \right]^{1/2}.
 \ee

 Moreover, \eqref{ineqTiti2} and \eqref{boundqH1} imply
 \begin{multline}\label{est6Fourier}
  |(B(\bq, \bp), A\bw)_{L^2}| \leq c_T \|\bq\|_{H^1} \|\bp\|_{H^1} |A\bw|_{L^2} \left[ 1 + \log\left( \frac{|A \bp|_{L^2}}{\lambda_1^{1/2}\|\bp\|_{H^1}}\right) \right]^{1/2} \\
  \leq c_T C_1  \frac{L_N^2}{\lambda_{N+1}^{1/2}} M_1 |A\bw|_{L^2} \leq \frac{\nu}{12}|A\bw|_{L^2}^2 + c_T^2 \frac{C_1^2}{\nu} \frac{L_N^4}{\lambda_{N+1}}M_1^2.
 \end{multline}

 Also, observe that
 \begin{eqnarray}\label{est7Fourier}
  -\beta \| P_K \bw\|_{H^1}^2 &=& - \beta \|\bw \|_{H^1}^2 + \beta \|Q_K \bw\|_{H^1}^2 \nonumber \\
  &\leq& - \beta \|\bw \|_{H^1}^2 + \frac{\beta}{\lambda_{K+1}} |A\bw|_{L^2}^2 \nonumber \\
  &\leq& - \beta \|\bw \|_{H^1}^2 + \frac{\nu}{2} |A\bw|_{L^2}^2,
 \end{eqnarray}
 where in the last inequality we used hypothesis \eqref{condKFourier}.

 Plugging estimates \eqref{est1Fourier}-\eqref{est7Fourier} into \eqref{enstrophyeqwFourier}, we obtain that
 \begin{multline}\label{enstrophyineqw1Fourier}
  \frac{\rd }{\rd t}\|\bw\|_{H^1}^2 + \frac{\nu}{2} |A\bw|_{L^2}^2 \leq -\beta \|\bw\|_{H^1}^2 \\
  + c M_1 \|\bw\|_{H^1} |A \bw|_{L^2} \left[ 1 + \log\left( \frac{|A \bw|_{L^2}}{\lambda_1^{1/2}\|\bw\|_{H^1}}\right) \right]^{1/2} \\
  + c \|\bw\|_{H^1}^2 |A \bw|_{L^2} \left[ 1 + \log\left( \frac{|A \bw|_{L^2}}{\lambda_1^{1/2}\|\bw\|_{H^1}}\right) \right]^{1/2} + c \frac{C_1^4}{\nu} \frac{L_N^6}{\lambda_{N+1}^2}
  + c \frac{C_1^2}{\nu} \frac{L_N^4}{\lambda_{N+1}}M_1^2.
 \end{multline}

 Since $\bv_N \in \mC([t_0, \infty);V)$ (\cite[Theorem 5]{AzouaniOlsonTiti2014}) and
 \[
  \|\bw(t_0)\|_{H^1} \leq \|P_N \bv_0\|_{H^1} + \|\bp(t_0)\|_{H^1} \leq 2 M_1,
 \]
 then there exists $\tau \in (t_0, \infty)$ such that
 \[
  \|\bw(t)\|_{H^1} \leq 3 M_1, \quad \forall t \in [t_0, \tau].
 \]

 Define
 \be\label{defttildeFourier}
  \tilde{t} = \sup\left\{ \tau \in (t_0, \infty) : \max_{t\in[t_0,\tau]} \|\bw(t)\|_{H^1} \leq 3 M_1 \right\}.
 \ee

 Suppose that $\tilde{t} < \infty$.

 Then, from \eqref{enstrophyineqw1Fourier}, we obtain that, for all $t \in [t_0, \tilde{t}]$,
 \begin{multline}\label{enstrophyineqw2Fourier}
  \frac{\rd }{\rd t}\|\bw\|_{H^1}^2 + \frac{\nu}{2} |A\bw|_{L^2}^2 \leq -\beta \|\bw\|_{H^1}^2 \\
  + c M_1 \|\bw\|_{H^1} |A \bw|_{L^2} \left[ 1 + \log\left( \frac{|A \bw|_{L^2}}{\lambda_1^{1/2}\|\bw\|_{H^1}}\right) \right]^{1/2}
  + c \frac{C_1^4}{\nu} \frac{L_N^6}{\lambda_{N+1}^2}
  + c \frac{C_1^2}{\nu} \frac{L_N^4}{\lambda_{N+1}}M_1^2.
 \end{multline}

 Observe that
 \begin{multline}\label{absorbingtermsFourier}
  \frac{\nu}{4} |A\bw|_{L^2}^2 - c M_1 \|\bw\|_{H^1} |A\bw|_{L^2} \left( 1 + \log\left( \frac{|A\bw|_{L^2}}{\lambda_1^{1/2} \|\bw\|_{H^1}} \right) \right)^{1/2} + \frac{\beta}{2} \|\bw\|_{H^1}^2 \\
  = \frac{\nu \lambda_1}{4} \|\bw\|_{H^1}^2 \left[ \frac{|A\bw|_{L^2}^2}{\lambda_1 \|\bw\|_{H^1}^2} - c \frac{M_1}{\nu \lambda_1^{1/2}} \frac{|A\bw|_{L^2}}{\lambda_1^{1/2} \|\bw\|_{H^1}}  \left( 1 + \log\left( \frac{|A\bw|_{L^2}^2}{\lambda_1 \|\bw\|_{H^1}^2} \right) \right)^{1/2}\right.\\
 \left. + \frac{2\beta}{\nu \lambda_1} \right].
 \end{multline}

 Define
 \be\label{defphiFourier}
  \phi(r) = r^2 - \rho r ( 1 + \log (r^2) )^{1/2} + B, \quad r \geq 1,
 \ee
 where
 \be\label{defrhoBFourier}
  \rho = c \frac{M_1}{\nu \lambda_1^{1/2}}, \quad B = \frac{2\beta}{\nu \lambda_1}.
 \ee
 Notice that
 \be\label{phiFourier}
 \phi(r) = \frac{r( \widetilde{\phi}(r^2) + B) + \rho(1 + \log (r^2))^{1/2}}{r + \rho(1 + \log (r^2))^{1/2}},
 \ee
 where
 \be\label{defphitildeFourier}
 \widetilde{\phi}(r) = r - {\rho}^2(1 + \log r).
 \ee

 One easily verifies that
 \be\label{minphiFourier}
 \min_{r\geq 1} \widetilde{\phi}(r) \geq -{\rho}^2 \log ({\rho}^2).
 \ee

 Thus, from \eqref{phiFourier} and \eqref{minphiFourier}, it follows that if
 \be\label{condrhoBFourier}
  B \geq {\rho}^2 \log ({\rho}^2),
 \ee
 then
 \be\label{phipositiveFourier}
  \phi(r) \geq 0, \quad \forall r \geq 1.
 \ee

 Now, by the definition of $\rho$ and $B$ in \eqref{defrhoBFourier}, we see that \eqref{condrhoBFourier} follows from hypothesis \eqref{condbetaFourier} on $\beta$.

 Using the fact \eqref{phipositiveFourier} with
 \[
  r =  \frac{|A\bw|_{L^2}}{\lambda_1^{1/2} \|\bw\|_{H^1}} \geq 1,
 \]
 we conclude that the right-hand side of \eqref{absorbingtermsFourier} is non-negative. Thus, from \eqref{enstrophyineqw2Fourier}, it follows that
 \be\label{enstrophyineqw3Fourier}
   \frac{\rd }{\rd t}\|\bw\|_{H^1}^2 + \frac{\nu}{4} |A\bw|_{L^2}^2 \leq -\frac{\beta}{2} \|\bw\|_{H^1}^2
   + c \frac{C_1^4}{\nu} \frac{L_N^6}{\lambda_{N+1}^2} + c \frac{C_1^2}{\nu} \frac{L_N^4}{\lambda_{N+1}}M_1^2.
 \ee

 Ignoring the second term on the left-hand side of \eqref{enstrophyineqw3Fourier} and integrating from $t_0$ to $t \in [t_0,\tilde{t}]$, we obtain that
 \begin{multline}\label{enstrophyineqw4Fourier}
  \|\bw(t)\|_{H^1}^2 \leq \|\bw(t_0)\|_{H^1}^2 \Exp^{-\frac{\beta}{2}(t-t_0)} \\
  + \frac{c}{\beta} \left[ \frac{C_1^4}{\nu} \frac{L_N^6}{\lambda_{N+1}^2} + \frac{C_1^2}{\nu} \frac{L_N^4}{\lambda_{N+1}}M_1^2 \right] (1 - \Exp^{-\frac{\beta}{2}(t-t_0)}).
 \end{multline}

 Notice that the functions
 \[
  f_1(x) = \frac{(1 + \log x)^3}{x^2}, \quad f_2(x) = \frac{(1 + \log x)^2}{x}
 \]
are both decreasing for $x \geq \Exp$. Since $N \geq K$ and, by hypotheses \eqref{condmFourier}, \eqref{condbetaFourier} and \eqref{condKFourier}, we have
\[
 \frac{\lambda_{N+1}}{\lambda_1} \geq \frac{\lambda_{K+1}}{\lambda_1} \geq \frac{2 \beta}{\nu\lambda_1} \geq \frac{2\lambda_m}{\lambda_1} \geq \Exp,
\]
it then follows that
\be\label{estquotientLNlambdaN1}
 \frac{L_N^6}{\lambda_{N+1}^2} \leq c \frac{L_m^6}{\lambda_m^2},
\ee
and
\be\label{estquotientLNlambdaN2}
 \frac{L_N^4}{\lambda_{N+1}} \leq c \frac{L_m^4}{\lambda_m}.
\ee

Plugging \eqref{estquotientLNlambdaN1} and \eqref{estquotientLNlambdaN2} into \eqref{enstrophyineqw4Fourier} and using hypothesis \eqref{condmFourier} with a suitable absolute constant $c$, we obtain that
 \be
  \|\bw(t)\|_{H^1}^2 \leq \|\bw(t_0)\|_{H^1}^2 \Exp^{-\frac{\beta}{2}(t-t_0)} + 4 M_1^2 (1 - \Exp^{-\frac{\beta}{2}(t-t_0)}) \leq 4 M_1^2, \quad \forall t \in [t_0, \tilde{t}].
 \ee
 Thus,
 \be
  \|\bw(t)\|_{H^1} \leq 2 M_1, \quad \forall t \in [t_0, \tilde{t}].
 \ee
 In particular, $\|\bw(\tilde{t})\|_{H^1} \leq 2 M_1$, which, by the definition of $\tilde{t}$ and the fact that $\bw \in \mC([t_0, \infty);V)$, contradicts the assumption that $\tilde{t} < \infty$. Therefore, the above argument implies
 \be
  \|\bw(t)\|_{H^1} \leq 2 M_1, \quad \forall t \geq t_0.
 \ee
\end{proof}


Next, we present a technical lemma.

\begin{lem}\label{lemesty}
Assume that $y: [t_0,\infty) \rightarrow [0,\infty)$ is a continuous function satisfying
\be\label{lemesty1}
y(s) \leq a \Exp^{-b (s - t)} y(t) + \gamma \sup_{t \leq \tau \leq s} y(\tau) + \varepsilon, \quad \forall  s \geq t \geq t_0,
\ee
with $\varepsilon \geq 0$, $a \geq 0$, $b > 0$ and $\gamma \in (0,1)$ such that
\be\label{condtheta}
 \theta = a \left( \Exp^{-\frac{b}{\nu \lambda_1}} + \frac{\gamma}{1 - \gamma}\right) < 1.
\ee
Then,
\be\label{lemesty1a}
y(t) \leq a \frac{\theta^{(t - t_0) \nu \lambda_1 -1}}{1 - \gamma} y(t_0) + \left( \frac{a}{(1-\theta)(1-\gamma)} + 1 \right) \frac{\varepsilon}{1 - \gamma}, \quad \forall t \geq t_0.
\ee
\end{lem}
\begin{proof}
Taking the sup on both sides of \eqref{lemesty1} over $s \in [t, t + (\nu \lambda_1)^{-1}]$, it follows that
\[ \sup_{t \leq s \leq t + (\nu \lambda_1)^{-1}} y(s) \leq a y(t) + \gamma \sup_{t \leq \tau \leq t + (\nu \lambda_1)^{-1}} y(\tau)  + \varepsilon.
\]
Thus,
\be\label{lemestsup}
 \sup_{t \leq \tau \leq t + (\nu \lambda_1)^{-1}} y(\tau)  \leq \frac{a}{1 - \gamma} y(t) + \frac{\varepsilon}{1 - \gamma}.
\ee
Using \eqref{lemestsup} in \eqref{lemesty1} with $s = t + (\nu \lambda_1)^{-1}$, $t \geq t_0$, yields
\be\label{lemesty2}
 y(t + (\nu \lambda_1)^{-1}) \leq \theta y(t) + \frac{\varepsilon}{1 - \gamma},
\ee
with $\theta$ as defined in \eqref{condtheta}.

For each $n \in \mathbb{N}$, let $t_n = t_0 + n(\nu\lambda_1)^{-1}$. Since \eqref{lemesty2} is valid for every $t \geq t_0$, in particular,
\be\label{lemesty3}
 y(t_n) = y(t_{n-1} + (\nu\lambda_1)^{-1}) \leq \theta y(t_{n-1}) + \frac{\varepsilon}{1 - \gamma}, \quad \forall n \in \mathbb{N}.
\ee

Hence, by induction, one has
\be\label{lemesty4}
 y(t_n) \leq \theta^n y(t_0) + \frac{\varepsilon}{(1-\theta)(1-\gamma)}, \quad \forall n \in \mathbb{N}.
\ee

Using \eqref{lemesty4} in \eqref{lemestsup} with $t = t_n$, it follows that
\be\label{lemesty4a}
 \sup_{t_n \leq s \leq t_{n+1}} y(s) \leq a\frac{\theta^n}{1 - \gamma} y(t_0) + \left( \frac{a}{(1-\theta)(1-\gamma)}+1\right) \frac{\varepsilon}{1-\gamma}.
\ee

Notice that, for every $t \in [t_n, t_{n+1}]$,
\be\label{eqn}
 n = (t_{n+1} - t_0) \nu \lambda_1 -1 \geq (t - t_0) \nu \lambda_1 -1.
\ee

Since $\theta \in [0,1)$, by hypothesis \eqref{condtheta}, it then follows from \eqref{lemesty4a} and \eqref{eqn} that, for every $t \in [t_n,t_{n+1}]$,
\be\label{lemesty5}
 y(t) \leq \sup_{t_n \leq s \leq t_{n+1}} y(s) \leq a\frac{\theta^{(t - t_0) \nu \lambda_1 -1}}{1 - \gamma} y(t_0) + \left( \frac{a}{(1-\theta)(1-\gamma)}+1\right) \frac{\varepsilon}{1-\gamma}.
\ee

Since \eqref{lemesty5} is valid for any $n \in \mathbb{N}$, \eqref{lemesty1a} follows.
\end{proof}

The following proposition is a direct consequence of the result proved in \cite[Lemma 1]{GarciaNovoTiti1999} (see also \cite{Titi1990}).

\begin{prop}\label{lemboundG}
 Let $\bu$ be a solution of \eqref{NSEeqs} satisfying \eqref{unifboundsattractor}-\eqref{boundqH1}, for every $t \geq t_0$. Then, the following inequalities hold
 \be
  |A^{-1} P_N B(\bp, \bq)|_{L^2},\, |A^{-1} P_N B(\bq,\bp)|_{L^2} \leq c M_1 L_N \|\bq\|_{V'},
 \ee
 \be
  |A^{-1} P_N B(\bq,\bq)|_{L^2} \leq c L_N |\bq|_{L^2}^2.
 \ee
\end{prop}

Using the results of Lemma \ref{lemesty} and Propositions \ref{thmUnifBoundvNH1Fourier} and \ref{lemboundG}, we can now obtain a uniform in time estimate for $|\bv_N(t) - \bp(t)|_{L^2}$. The proof below follows similar ideas to the proof of \cite[Theorem 2]{GarciaNovoTiti1999}. We use the notation $\|\cdot\|_{\mL(X)}$ to denote the operator norm in the space $\mL(X)$, the space of bounded linear operators on a Hilbert space $X$.

\begin{thm}\label{thmestvNPNu}
 Let $\bu$ be a solution of \eqref{NSEeqs} satisfying \eqref{unifboundsattractor}-\eqref{boundqH1}, for every $t \geq t_0$. Let $\bv_0 \in B_V(M_1)$, with $M_1$ as in \eqref{unifboundsattractor}. For every $N \in \mathbb{N}$, let $\bv_N$ be the unique solution of \eqref{eqDataAssFourierGalerkin} satisfying $\bv_N(t_0) = P_N \bv_0$. Fix $\alpha \in (1/2,1)$ and consider $m \in \mathbb{N}$ large enough such that
 \be\label{condmFourier2}
  \lambda_m \geq \max \left\{ \frac{\lambda_1 \Exp}{2}, c \frac{C_1}{\nu}L_m^2, c \left( \frac{C_1^2}{\nu M_1}\right)^{2/3} L_m^2, \left[ c c_\alpha \left( 1 + \frac{\Exp^{-\alpha}}{1-\alpha}  \right)\frac{|\Omega|^{\alpha - \frac{1}{2}} M_1}{\nu}\right]^{\frac{1}{1-\alpha}} \right\},
 \ee
 where $c_\alpha$ is the constant from \eqref{ineqA-alphadiffBuBv}.

 If $\beta > 0$ and $K \in \mathbb{N}$ are large enough such that
 \be\label{condbetaFourier2}
  \beta \geq \max\left\{ \nu \lambda_m, c \frac{M_1^2}{\nu}\left[ 1 + \log \left(\frac{M_1}{\nu\lambda_1^{1/2}}\right)\right] \right\}
 \ee
 and
 \be\label{condKFourier2}
  \lambda_{K+1} \geq \frac{2\beta}{\nu},
 \ee
 then, there exists $\theta = \theta(\beta) \in [0,1)$ and a constant $C = C(\nu, \lambda_1, |\bg|_{L^2})$ such that, for every $N \geq K$,
 \be\label{estsupwL2}
  |\bv_N(t) -\bp(t)|_{L^2} \leq c \theta^{(t - t_0) \nu \lambda_1 -1} |\bv_N(t_0) -\bp(t_0)|_{L^2} + C \frac{L_N^4}{\lambda_{N+1}^{3/2}}.
 \ee
\end{thm}
\begin{proof}
 Denote $\bw = \bv_N - \bp$. Subtracting \eqref{eqNSEprojPN} from \eqref{eqDataAssFourierGalerkin}, yields
 \begin{multline}\label{eqwThmFourier}
  \frac{\rd \bw}{\rd t} + \nu A \bw = - P_N [ B(\bv_N,\bv_N) - B(\bu,\bu)] -\beta P_K\bw \\
  = - P_N [ B(\bv_N,\bv_N) - B(\bp,\bp)] + P_N G - \beta P_K \bw,
 \end{multline}
 where
 \be
  G(t) = B(\bu(t),\bu(t)) - B(\bp(t),\bp(t)), \quad \forall t \geq t_0.
 \ee

 Using that $P_N \bw = \bw $, we can also rewrite \eqref{eqwThmFourier} as
 \be
  \frac{\rd \bw}{\rd t} + [\nu A P_N + \beta P_K] \bw = - P_N [ B(\bv_N,\bv_N) - B(\bp,\bp)] + P_N G.
 \ee

 Using Duhamel's formula, it follows that, for every $s \geq t \geq t_0$,
 \begin{multline}\label{ineqwDuhamelFourier}
  |\bw(s)|_{L^2} \leq |\Exp^{-(s - t) (\nu A P_N + \beta P_K)} \bw(t)|_{L^2} \\
  + \int_{t}^s \left|\Exp^{-(s-\tau) (\nu A P_N + \beta P_K)} P_N [ B(\bv_N(\tau),\bv_N(\tau)) - B(\bp(\tau),\bp(\tau))] \right|_{L^2} \rd \tau \\
  + \int_{t}^s \left| \Exp^{-(s-\tau) (\nu A P_N + \beta P_K)} P_N G(\tau) \right|_{L^2} \rd \tau.
 \end{multline}

 We now estimate each term on the right-hand side of \eqref{ineqwDuhamelFourier}.

 Notice that, for every $s \geq t \geq t_0$,
 \begin{multline}\label{estinitcondFourier}
  |\Exp^{-(s-t) (\nu A P_N + \beta P_K)} \bw(t)|_{L^2} \leq \\
  \leq (\|\Exp^{-(s-t) (\nu A P_K + \beta P_K)}\|_{\mL(P_K H)} + \|\Exp^{-(s-t) \nu A P_N Q_K}\|_{\mL(P_N Q_K H)}) |\bw(t)|_{L^2}  \\
  = \left[ \left( \max_{1 \leq j \leq K} \Exp^{-(s-t)(\nu \lambda_j + \beta)}\right) +  \left( \max_{K+1 \leq j \leq N} \Exp^{-(s-t)\nu \lambda_j}\right) \right] |\bw(t)|_{L^2} \\
  = \left( \Exp^{-(s-t)(\nu\lambda_1 + \beta)} + \Exp^{-(s-t)(\nu\lambda_{K+1})} \right) |\bw(t)|_{L^2} \\
  \leq 2 \Exp^{-(s-t)\beta} |\bw(t)|_{L^2},
 \end{multline}
where in the last inequality we used that $\nu \lambda_{K+1} \geq 2 \beta$, from hypothesis \eqref{condKFourier2}.

 Using \eqref{ineqA-alphadiffBuBv}, we obtain that
 \begin{multline}\label{estExpdiffBvNp}
   \left|\Exp^{-(s-\tau) (\nu A P_N + \beta P_K)} P_N [ B(\bv_N(\tau),\bv_N(\tau)) - B(\bp(\tau),\bp(\tau))] \right|_{L^2} \\
  = \frac{1}{\nu^\alpha}  \left|\nu^\alpha A^\alpha \Exp^{-(s-\tau) (\nu A P_N + \beta P_K)} A^{-\alpha} P_N [ B(\bv_N(\tau),\bv_N(\tau)) - B(\bp(\tau),\bp(\tau))] \right|_{L^2} \\
  \leq c_\alpha \frac{|\Omega|^{\alpha - \frac{1}{2}}}{\nu^\alpha} \|\nu^\alpha A^\alpha \Exp^{-(s-\tau) (\nu A P_N + \beta P_K)}\|_{\mL(P_N H)} \|\bv_N(\tau) + \bp(\tau)\|_{H^1} |\bw(\tau)|_{L^2}.
 \end{multline}

 By Proposition \ref{thmUnifBoundvNH1Fourier}, we have that
 \be
  \|\bv_N(\tau) + \bp(\tau)\|_{H^1} \leq 3 M_1, \quad \forall \tau \geq t_0, \quad \forall N \geq K.
 \ee

It then follows from \eqref{estExpdiffBvNp} that
 \begin{multline}
  \left|\Exp^{-(s-\tau) (\nu A P_N + \beta P_K)} P_N [ B(\bv_N(\tau),\bv_N(\tau)) - B(\bp(\tau),\bp(\tau))] \right|_{L^2} \\
  \leq 3 c_\alpha \frac{|\Omega|^{\alpha - \frac{1}{2}} M_1}{\nu^\alpha} \|\nu^\alpha A^\alpha \Exp^{-(s-\tau) (\nu A P_N + \beta P_K)}\|_{\mL(P_N H)} |\bw(\tau)|_{L^2}.
 \end{multline}

 Notice that, by using hypotheses \eqref{condbetaFourier2} and \eqref{condKFourier2}, we have
 \be\label{ineqKgeqm}
  \frac{\nu \lambda_{K+1}}{2} \geq \beta \geq \nu \lambda_m,
 \ee
 which implies in particular that $K \geq m$.

 Now, we write
 \begin{multline}\label{separatingExp}
  \Exp^{-(s-\tau) (\nu A P_N + \beta P_K)} = \Exp^{-(s-\tau) (\nu A P_m + \beta P_m)} + \Exp^{-(s-\tau) (\nu A P_K Q_m + \beta P_K Q_m)} \\
  + \Exp^{-(s-\tau) \nu A P_N Q_K}.
 \end{multline}

 Therefore,
 \begin{multline}\label{estdiffnonlineartermsFourier}
  \int_{t}^s \left|\Exp^{-(s-\tau) (\nu A P_N + \beta P_K)} P_N [ B(\bv_N(\tau),\bv_N(\tau)) - B(\bp(\tau),\bp(\tau))] \right|_{L^2} \rd \tau \\
  \leq 3c_\alpha\frac{ |\Omega|^{\alpha - \frac{1}{2}} M_1}{\nu^\alpha} \left( \sup_{t \leq \tau \leq s} |\bw(\tau)|_{L^2} \right) \int_{t}^s \|\nu^\alpha A^\alpha \Exp^{-(s-\tau) (\nu A P_N + \beta P_K)}\|_{\mL(P_N H)} \rd \tau \\
  \leq 3c_\alpha \frac{ |\Omega|^{\alpha - \frac{1}{2}} M_1}{\nu^\alpha} \left( \sup_{t \leq \tau \leq s} |\bw(\tau)|_{L^2} \right) \left( \int_{t}^s \|\nu^\alpha A^\alpha \Exp^{-(\xi-t) (\nu A P_m + \beta P_m)}\|_{\mL(P_m H)} \rd \xi + \right. \\
  + \int_{t}^s \|\nu^\alpha A^\alpha \Exp^{-(\xi-t) (\nu A P_K Q_m + \beta P_K Q_m)}\|_{\mL(P_K H)} \rd \xi \\
  \left. + \int_{t}^s \|\nu^\alpha A^\alpha \Exp^{-(\xi-t) \nu A P_N Q_K }\|_{\mL(P_N Q_K H)} \rd \xi \right),
 \end{multline}
 where in the second inequality we used \eqref{separatingExp} and applied the change of variables $\xi = s - \tau + t$.

 Notice that
 \begin{multline}\label{opnormPmHFourier}
  \| \nu^\alpha A^\alpha \Exp^{-(\xi - t)(\nu A P_m + \beta P_m)}\|_{\mL(P_m H)} = \max_{1 \leq j \leq m} (\nu \lambda_j)^\alpha \Exp^{- (\xi - t)(\nu \lambda_j + \beta)} \\
  \leq \Exp^{- (\xi - t) \beta} \max_{\nu\lambda_1 \leq x \leq \nu \lambda_m} x^\alpha \Exp^{- (\xi - t) x} \\
  = \Exp^{- (\xi - t) \beta} \cdot \left\{ \begin{array}{lll}
                                                                                                        \displaystyle (\nu \lambda_m)^\alpha \Exp^{- (\xi - t) \nu \lambda_m}, &\mbox{ if } \xi < t + \frac{\alpha}{\nu \lambda_m}, \vspace{0.2cm} \\ 
                                                                                                        \displaystyle \frac{\alpha^\alpha}{ (\xi - t)^\alpha} \Exp^{-\alpha}, &\mbox{ if } t + \frac{\alpha}{\nu \lambda_m} \leq \xi \leq t + \frac{\alpha}{\nu \lambda_1}, \vspace{0.2cm} \\
                                                                                                        \displaystyle (\nu \lambda_1)^\alpha \Exp^{- (\xi - t) \nu \lambda_1}, &\mbox{ if } \xi > t + \frac{\alpha}{\nu \lambda_1}.
                                                                                                       \end{array}\right.
 \end{multline}

 Let us decompose $[t,s]$ as the union of the intervals
 \begin{multline}
  I_1 = \left[t, t + \frac{\alpha}{\nu \lambda_m}\right] \cap [t, s], \quad I_2 = \left[ t + \frac{\alpha}{\nu \lambda_m}, t + \frac{\alpha}{\nu \lambda_1}\right] \cap [t, s],\\
  I_3 = \left[ t + \frac{\alpha}{\nu \lambda_1}, \infty\right) \cap [t, s].
 \end{multline}

 We then have
 \begin{multline}\label{ineqintI1Fourier}
  \int_{I_1}  \| \nu^\alpha A^\alpha \Exp^{-(\xi - t)(\nu A P_m + \beta P_m)}\|_{\mL(P_m H)} \rd \xi \leq \\
  \leq \int_{t}^{t + \frac{\alpha}{\nu \lambda_m}} (\nu \lambda_m)^\alpha \Exp^{- (\xi - t) (\nu \lambda_m + \beta)} \rd \xi = \frac{(\nu \lambda_m)^\alpha}{\nu \lambda_m + \beta} (1 - \Exp^{-\alpha} \Exp^{-\frac{\alpha \beta}{\nu \lambda_m}}),
 \end{multline}
 \begin{multline}\label{ineqintI2Fourier}
  \int_{I_2}  \| \nu^\alpha A^\alpha \Exp^{-(\xi - t)(\nu A P_m + \beta P_m)}\|_{\mL(P_m H)} \rd \xi \leq \\
  \leq \int_{t + \frac{\alpha}{\nu \lambda_m}}^{t + \frac{\alpha}{\nu \lambda_1}} \frac{\alpha^\alpha}{ (\xi - t)^\alpha} \Exp^{-\alpha} \Exp^{- (\xi - t) \beta} \rd \xi \leq (\nu \lambda_m)^\alpha \Exp^{-\alpha} \int_{t + \frac{\alpha}{\nu \lambda_m}}^{t + \frac{\alpha}{\nu \lambda_1}}  \Exp^{- (\xi - t) \beta} \rd \xi \\
  = \frac{(\nu \lambda_m)^\alpha \Exp^{-\alpha}}{\beta} ( \Exp^{-\frac{\alpha\beta}{\nu \lambda_m}} - \Exp^{-\frac{\alpha\beta}{\nu \lambda_1}})
 \end{multline}
 and
 \begin{multline}\label{ineqintI3Fourier}
  \int_{I_3}  \| \nu^\alpha A^\alpha \Exp^{-(\xi - t)(\nu A P_m + \beta P_m)}\|_{\mL(P_m H)} \rd \xi \leq \int_{t + \frac{\alpha}{\nu \lambda_1}}^\infty (\nu \lambda_1)^\alpha \Exp^{-(\xi-t)(\nu \lambda_1 + \beta)} \rd \xi \\
  = \frac{(\nu\lambda_1)^\alpha}{\nu \lambda_1 + \beta} \Exp^{-\alpha} \Exp^{-\frac{\alpha\beta }{\nu \lambda_1}}.
 \end{multline}

 Notice that the estimate in \eqref{ineqintI3Fourier} is smaller than the absolute value of the negative term in \eqref{ineqintI2Fourier}. Thus, from \eqref{ineqintI1Fourier}-\eqref{ineqintI3Fourier}, it follows that
 \begin{multline}\label{intopnormPmHFourier}
  \int_{t}^s \| \nu^\alpha A^\alpha \Exp^{-(\xi - t)(\nu A P_m + \beta P_m)}\|_{\mL(P_m H)} \rd \xi \leq \\
  \leq \frac{(\nu \lambda_m)^\alpha}{\nu \lambda_m + \beta} (1 - \Exp^{-\alpha} \Exp^{-\frac{\alpha \beta}{\nu \lambda_m}}) + \frac{(\nu \lambda_m)^\alpha}{\beta} \Exp^{-\alpha} \Exp^{-\frac{\alpha\beta}{\nu \lambda_m}} \leq \frac{(\nu \lambda_m)^\alpha}{\beta}.
 \end{multline}

 Now, similarly as in \eqref{opnormPmHFourier}, we have that
 \begin{multline}\label{opnormPKQmHFourier}
  \| \nu^\alpha A^\alpha \Exp^{-(\xi - t) (\nu A P_K Q_m + \beta P_K Q_m)}\|_{\mL(P_K Q_m H)} \leq \\
  \leq \Exp^{-(\xi-t)\beta} \cdot \left\{ \begin{array}{lll}
                                                                                                        \displaystyle (\nu \lambda_K)^\alpha \Exp^{- (\xi - t) \nu \lambda_K}, &\mbox{ if } \xi < t + \frac{\alpha}{\nu \lambda_K}, \vspace{0.2cm} \\ 
                                                                                                        \displaystyle \frac{\alpha^\alpha}{ (\xi - t)^\alpha} \Exp^{-\alpha}, &\mbox{ if } t + \frac{\alpha}{\nu \lambda_K} \leq \xi \leq t + \frac{\alpha}{\nu \lambda_{m+1}}, \vspace{0.2cm} \\
                                                                                                        \displaystyle (\nu \lambda_{m+1})^\alpha \Exp^{- (\xi - t) \nu \lambda_{m+1}}, &\mbox{ if } \xi > t + \frac{\alpha}{\nu \lambda_{m+1}}.
                                                                                                       \end{array}\right.
 \end{multline}

 We decompose $[t,s]$ as the union of the intervals
 \begin{multline}
  J_1 = \left[t, t + \frac{\alpha}{\nu \lambda_K}\right] \cap [t, s], \quad J_2 = \left[ t + \frac{\alpha}{\nu \lambda_K}, t + \frac{\alpha}{\nu \lambda_{m+1}}\right] \cap [t,s],\\
  J_3 = \left[ t + \frac{\alpha}{\nu \lambda_{m+1}}, \infty\right) \cap [t, s].
 \end{multline}

 We have
 \begin{multline}\label{ineqintJ1Fourier}
 \int_{J_1} \| \nu^\alpha A^\alpha \Exp^{-(\xi - t)(\nu A P_K Q_m + \beta P_K Q_m)}\|_{\mL(P_K Q_m H)} \rd \xi \leq \\
 \leq \int_{t}^{t + \frac{\alpha}{\nu \lambda_K}} (\nu \lambda_K)^\alpha \Exp^{- (\xi - t) \nu \lambda_K } \rd \xi = \frac{1 - \Exp^{-\alpha}}{(\nu \lambda_K)^{1-\alpha}},
 \end{multline}
 \begin{multline}\label{ineqintJ2Fourier}
  \int_{J_2} \| \nu^\alpha A^\alpha \Exp^{-(\xi - t)(\nu A P_K Q_m + \beta P_K Q_m)}\|_{\mL(P_K Q_m H)} \rd \xi \leq \\
  \leq \int_{t + \frac{\alpha}{\nu \lambda_K}}^{t + \frac{\alpha}{\nu \lambda_{m+1}}} \frac{\alpha^\alpha}{ (\xi - t)^\alpha} \Exp^{-\alpha} \rd \xi = \frac{\alpha \Exp^{-\alpha}}{1-\alpha} \left( \frac{1}{(\nu \lambda_{m+1})^{1-\alpha}} -  \frac{1}{(\nu \lambda_{K})^{1-\alpha}} \right)
 \end{multline}
 and
 \begin{multline}\label{ineqintJ3Fourier}
  \int_{J_3} \| \nu^\alpha A^\alpha \Exp^{-(\xi - t) (\nu A P_K Q_m + \beta P_K Q_m)}\|_{\mL(P_K Q_m H)} \rd \xi \leq \\
  \leq \int_{t + \frac{\alpha}{\nu \lambda_{m+1}}}^\infty (\nu \lambda_{m+1})^\alpha \Exp^{-(\xi-t)\nu \lambda_{m+1}} \rd \xi = \frac{\Exp^{-\alpha}}{(\nu\lambda_{m+1})^{1-\alpha}}.
 \end{multline}

 Thus, summing up \eqref{ineqintJ1Fourier}-\eqref{ineqintJ3Fourier}, we obtain
 \begin{multline}\label{intopnormPNQmHFourier}
  \int_{t}^s  \| \nu^\alpha A^\alpha \Exp^{-(\xi - t)(\nu A P_K Q_m + \beta P_K Q_m) }\|_{\mL(P_K Q_m H)} \rd \xi \leq \\
  \leq \frac{1 - \Exp^{-\alpha}}{(\nu \lambda_K)^{1-\alpha}} + \frac{\alpha \Exp^{-\alpha}}{1-\alpha} \left( \frac{1}{(\nu \lambda_{m+1})^{1-\alpha}} -  \frac{1}{(\nu \lambda_{K})^{1-\alpha}} \right) + \frac{\Exp^{-\alpha}}{(\nu\lambda_{m+1})^{1-\alpha}} \\
  = \left( 1 - \frac{\Exp^{-\alpha}}{1-\alpha}\right)  \frac{1}{(\nu \lambda_{K})^{1-\alpha}} + \frac{\Exp^{-\alpha}}{1-\alpha} \frac{1}{(\nu \lambda_{m+1})^{1-\alpha}} \\
  < \left( \frac{\Exp^{-\alpha}}{1-\alpha}\right) \frac{1}{(\nu \lambda_{m+1})^{1-\alpha}},
 \end{multline}
 where in the last inequality we used the fact that
 \be
  1 - \frac{\Exp^{-\alpha}}{1-\alpha} < 0, \quad \forall \alpha > 0.
 \ee

 Moreover, analogously to \eqref{opnormPKQmHFourier}-\eqref{intopnormPNQmHFourier}, one obtains that
 \be\label{intopnormPNQKH}
  \int_{t}^s  \| \nu^\alpha A^\alpha \Exp^{-(\xi - t)\nu A P_N Q_K }\|_{\mL(P_N Q_K H)} \rd \xi < \left(\frac{\Exp^{-\alpha}}{1-\alpha}\right) \frac{1}{(\nu \lambda_{K+1})^{1-\alpha}}.
 \ee

 Now, let us estimate the third term on the right-hand side of \eqref{ineqwDuhamelFourier}.

 Notice that
 \begin{multline}\label{estthirdterm0Fourier}
  \int_{t}^s \left| \Exp^{-(s-\tau) (\nu A P_N + \beta P_K)} P_N G(\tau) \right|_{L^2} \rd \tau = \\
  = \frac{1}{\nu} \int_{t}^s \left|\nu A \Exp^{-(s-\tau) (\nu A P_N + \beta P_K)} A^{-1} P_N G(\tau) \right|_{L^2} \rd \tau \leq \\
  \leq \frac{1}{\nu} \int_{t}^s \| \nu A \Exp^{-(s-\tau) (\nu A P_N + \beta P_K)}\|_{\mL(P_N H)} |A^{-1} P_N G (\tau)|_{L^2} \rd \tau.
 \end{multline}

 By Proposition \ref{lemboundG}, it follows that
 \be
  |A^{-1} P_N G|_{L^2} \leq c M_1 L_N \|\bq\|_{V'} + c L_N |\bq|_{L^2}^2 \leq cM_1 L_N \frac{|\bq|_{L^2}}{\lambda_{N+1}^{1/2}} + c L_N |\bq|_{L^2}^2 .
 \ee
 Then, using \eqref{boundqL2}, yields
 \be\label{boundA-1PNGFourier}
  |A^{-1} P_N G|_{L^2} \leq c C_N,
 \ee
 where
 \be\label{defCNFourier}
  C_N = C_0 \frac{L_N^2}{\lambda_{N+1}^{3/2}} \left( M_1 + C_0  \frac{L_N}{\lambda_{N+1}^{1/2}} \right),
 \ee
 with $C_0$ as defined in \eqref{defC0}.

 Now, similarly as in \eqref{opnormPmHFourier}-\eqref{intopnormPmHFourier}, one obtains that
 \begin{multline}\label{intopnormPNHFourier}
  \int_{t}^s \|\nu A \Exp^{-(s - \tau)(\nu A P_N + \beta P_K)}\|_{\mL(P_N H)} \rd \tau \\
  = \int_{t}^s \|\nu A \Exp^{-(\xi-t)(\nu A P_N + \beta P_K)}\|_{\mL(P_N H)} \rd \xi \\
  \leq \frac{\nu \lambda_N}{\nu \lambda_N + \beta} (1 - \Exp^{-1}\Exp^{-\frac{\beta}{\nu \lambda_N}}) + \log\left( \frac{\lambda_N}{\lambda_1}\right) \Exp^{-1}\Exp^{-\frac{\beta}{\nu \lambda_N}} + \frac{\nu \lambda_1}{\nu \lambda_1 + \beta} \Exp^{-1}\Exp^{-\frac{\beta}{\nu \lambda_1}} \\
  \leq 1 + \log \left( \frac{\lambda_N}{\lambda_1} \right) = L_N^2.
 \end{multline}

 Hence, from \eqref{estthirdterm0Fourier}, \eqref{boundA-1PNGFourier} and \eqref{intopnormPNHFourier}, we have
 \be\label{estthirdtermFourier}
  \int_{t}^s \left| \Exp^{-(s-\tau) (\nu A P_N + \beta P_K)} P_N G(\tau) \right|_{L^2} \rd \tau \leq c \frac{C_N L_N^2}{\nu},
 \ee
 with $C_N$ as defined in \eqref{defCNFourier}.

 Now, plugging estimates \eqref{estinitcondFourier}, \eqref{estdiffnonlineartermsFourier}, \eqref{intopnormPmHFourier}, \eqref{intopnormPNQmHFourier}, \eqref{intopnormPNQKH} and \eqref{estthirdtermFourier} into \eqref{ineqwDuhamelFourier}, we obtain that, for all $s \geq t \geq t_0$,
 \begin{multline}\label{estwL2a}
  |\bw(s)|_{L^2} \leq 2 \Exp^{-(s-t)\beta} |\bw(t)|_{L^2} + \\
  + 3 c_\alpha \frac{|\Omega|^{\alpha - \frac{1}{2}} M_1}{\nu^\alpha} \left( \sup_{t \leq \tau \leq s} |\bw(\tau)|_{L^2}\right) \left[ \frac{(\nu \lambda_m)^\alpha}{\beta} + \frac{\Exp^{-\alpha}}{1-\alpha} \frac{1}{(\nu \lambda_{m+1})^{1-\alpha}} + \right. \\
  \left. + \frac{\Exp^{-\alpha}}{1-\alpha} \frac{1}{(\nu \lambda_{K+1})^{1-\alpha}} \right]
  + c\frac{C_N L_N^2}{\nu}.
 \end{multline}

 Since $K \geq m$ (cf. \eqref{ineqKgeqm}), we have that
 \be
  \frac{(\nu \lambda_m)^\alpha}{\beta} + \frac{\Exp^{-\alpha}}{1-\alpha} \frac{1}{(\nu \lambda_{m+1})^{1-\alpha}} + \frac{\Exp^{-\alpha}}{1-\alpha} \frac{1}{(\nu \lambda_{K+1})^{1-\alpha}} \leq \left( 1 + 2 \frac{\Exp^{-\alpha}}{1-\alpha} \right)\frac{1}{(\nu \lambda_m)^{1-\alpha}}
 \ee

 Hence, from \eqref{estwL2a}, we obtain that
 \begin{multline}\label{estwL2b}
  |\bw(s)|_{L^2} \leq 2 \Exp^{-(s-t)\beta} |\bw(t)|_{L^2} + \\
  + c c_\alpha \left( 1 + \frac{\Exp^{-\alpha}}{1-\alpha} \right) \frac{|\Omega|^{\alpha - \frac{1}{2}} M_1}{\nu \lambda_m^{1-\alpha}} \sup_{t \leq \tau \leq s} |\bw(\tau)|_{L^2} + c\frac{C_N L_N^2}{\nu}.
 \end{multline}

 Let
 \be\label{defgamma}
  \gamma = c c_\alpha \left( 1 + \frac{\Exp^{-\alpha}}{1-\alpha} \right) \frac{|\Omega|^{\alpha - \frac{1}{2}} M_1}{\nu \lambda_m^{1-\alpha}}
 \ee
and
\[
 \theta = 2 \left( \Exp^{-\frac{\beta}{\nu \lambda_1}} + \frac{\gamma}{1 - \gamma} \right).
\]

Using hypothesis \eqref{condmFourier2} with a suitable absolute constant $c$ and also hypothesis \eqref{condbetaFourier2}, we obtain that $\gamma < 1$ and  $\theta < 1$. Therefore, \eqref{estsupwL2} follows from \eqref{estwL2b} and Lemma \ref{lemesty} with $y = |\bw(\cdot)|_{L^2}$, $a=2$, $b = \beta$, $\gamma$ given in \eqref{defgamma} and $\varepsilon = c C_N L_N^2/\nu$.
\end{proof}

\begin{rmk}
 We notice that, by using an explicit form of the constant $c_\alpha$ from \eqref{ineqA-alphadiffBuBv} (see, e.g., \cite{BartuccelliGibbon2011,Lieb1983,Mizohata1973}), one could obtain an optimal choice of $\alpha$ by minimizing the coefficient of $\sup_{t \leq \tau \leq s} |\bw(\tau)|_{L^2}$ in \eqref{estwL2b} with respect to $\alpha$. Thus, in this case, the values of $\gamma$, $\theta$, and the condition \eqref{condmFourier2} on $\lambda_m$ would be given explicitly in terms of this optimal value of $\alpha$. However, we chose not to deal with these technical details here.
\end{rmk}

With the result of Theorem \ref{thmestvNPNu}, we can obtain an estimate for the error committed when applying the standard Galerkin method to \eqref{DataAssAlg} in order to obtain an approximation of the reference solution $\bu$ of \eqref{NSEeqs}. The proof follows as in \eqref{ordererrorSGML2}-\eqref{ordererrorSGMH1}.

\begin{cor}\label{corerrorSGM}
 Assume the hypotheses from Theorem \ref{thmestvNPNu}. Then, there exists $T = T(\nu,\lambda_1,|\bg|_{L^2},N) \geq t_0$ such that, for every $N \geq K$,
 \be
  \sup_{t\geq T} |\bv_N(t) - \bu(t)|_{L^2} \leq C\frac{L_N}{\lambda_{N+1}},
 \ee
 and
 \be
  \sup_{t \geq T} \|\bv_N(t) - \bu(t)\|_{H^1} \leq C\frac{L_N}{\lambda_{N+1}^{1/2}},
 \ee
 where $C$ is a constant depending on $\nu$, $\lambda_1$ and $|\bg|_{L^2}$, but independent of $N$.
\end{cor}

Finally, we now state the result about the error associated with the Postprocessing Galerkin method applied to \eqref{eqDataAssFourier}, relative to the reference solution $\bu$. Compared to the result from Corollary \ref{corerrorSGM}, the estimates show that the Postprocessing Galerkin method has a better convergence rate than the standard Galerkin method. The proof follows immediately from the result of Theorem \ref{thmestvNPNu} and \eqref{distuPhi1L2b}-\eqref{inversePoincarevNPNu}.

\begin{thm}\label{thmerrorPPGM}
Assume the hypotheses from Theorem \ref{thmestvNPNu}, with $\bu$ satisfying, in addition, \eqref{distuPhi1L2aa} and \eqref{distuPhi1H1aa}, for every $t \geq t_0$. Then, there exists $T = T(\nu,\lambda_1,|\bg|_{L^2},N) \geq t_0$ such that, for every $N \geq K$,
\be\label{esterrorPPGML2}
 \sup_{t \geq T}|[\bv_N(t) + \Phi_1(\bv_N(t))] - \bu(t)|_{L^2} \leq C \frac{L_N^4}{\lambda_{N+1}^{3/2}},
\ee
and
\be\label{esterrorPPGMH1}
 \sup_{t \geq T}\|[\bv_N(t) + \Phi_1(\bv_N(t))] - \bu(t)\|_{H^1} \leq C \frac{L_N^4}{\lambda_{N+1}},
\ee
where $C$ is a constant depending on $\nu$, $\lambda_1$ and $|\bg|_{L^2}$, but independent of $N$.
\end{thm}
%

\subsection{A general class of interpolant operators}\label{subsecAnotherClassIntOp}

We now consider the class of linear interpolant operators $I_h : L^2(\Omega)^2 \to L^2(\Omega)^2$ satisfying the following properties:

\renewcommand{\theenumi}{{P}\arabic{enumi}}
\begin{enumerate}
 \item\label{propintP1} There exists a positive constant $c_0$ such that
 \be\label{eqpropintP1}
 |\varphi - I_h(\varphi)|_{L^2} \leq c_0 h \|\varphi\|_{H^1}, \quad \forall \varphi \in H^1(\Omega)^2.
\ee
 \item\label{propintP2} There exists a positive constant $c_{-1}$ such that
 \be\label{eqpropintP2}
 \|\varphi - I_h(\varphi)\|_{H^{-1}} \leq c_{-1} h |\varphi|_{L^2}, \quad \forall \varphi \in L^2(\Omega)^2.
\ee
 \item\label{propintP3} There exists a positive constant $\widetilde{c_0}$ such that
 \be\label{ineqpropintP3}
   |I_h(\bq)|_{L^2} \leq \widetilde{c_0} \frac{|\Omega|^{3/4}}{h^2 \lambda_{N+1}^{1/4}} |\bq|_{L^2}, \quad \forall \bq \in Q_N H.
 \ee
\end{enumerate}

As one easily verifies, the example of interpolant operator given by the low Fourier modes projector $P_K$, $N \geq K$, considered in subsection \ref{subsecFouriermodescase}, satisfies properties \eqref{propintP1}-\eqref{propintP3}. In particular, property \eqref{propintP3} is immediately verified, since $I_h(\bq) = P_K \bq = 0$. Indeed, the only reason for assuming property \eqref{propintP3} is that, as will be clearer in the proof of Theorem \ref{thmestvNPNuIh}, we do not assume $P_\sigma I_h$ to commute with $A$, a property that $P_K$ satisfies. This is the key difference between the proofs of Theorems \ref{thmestvNPNu} and \ref{thmestvNPNuIh}.

A more physically interesting example of operator $I_h$ satisfying properties \eqref{propintP1}-\eqref{propintP3} is given by local averages over finite volume elements. For illustrational purposes, this is proved in the Appendix.

The next results follow a similar outline from the ones in subsection \ref{subsecFouriermodescase}. We again assume either periodic or no-slip Dirichlet boundary conditions. As before, we start by obtaining a uniform estimate of the $V$ norm of $\bv_N - \bp$.

\begin{prop}\label{thmUnifBoundvNH1}
 Let $\bu$ be a solution of \eqref{NSEeqs} satisfying \eqref{unifboundsattractor}-\eqref{boundqH1}, for every $t \geq t_0$. Let $\bv_0 \in B_V(M_1)$, with $M_1$ as in \eqref{unifboundsattractor}. For every $N \in \mathbb{N}$, let $\bv_N$ be the unique solution of \eqref{eqDataAssFourierGalerkin} satisfying $\bv_N(t_0) = P_N \bv_0$.
 Consider $m \in \mathbb{N}$ large enough such that
 \be\label{condm}
  \lambda_m \geq \max \left\{ \frac{\lambda_1 \Exp}{2}, c \frac{C_1}{\nu}L_m^2, c \left( \frac{C_1^2}{\nu M_1}\right)^{2/3} L_m^2, c \left(\frac{C_1}{M_1}\right)^2 L_m^2 \right\}.
 \ee
 If $\beta > 0$ is large enough such that
 \be\label{condbeta}
  \beta \geq \max\left\{ \nu \lambda_m, c \frac{M_1^2}{\nu}\left[ 1 + \log \left(\frac{M_1}{\nu\lambda_1^{1/2}}\right)\right] \right\}
 \ee
 and if $h$ is small enough such that
 \be\label{condh}
  h \leq \frac{1}{c_0} \left( \frac{\nu}{\beta}\right)^{1/2},
 \ee
 where $c_0$ is the constant from \eqref{eqpropintP1}, then, for every $N \geq m$, we have
 \be\label{unifboundvN}
  \sup_{t \geq t_0} \|\bv_N(t) - \bp(t)\|_{H^1} \leq 2 M_1.
 \ee
\end{prop}
\begin{proof}
 Denote $\bw = \bv_N - \bp$. Subtracting \eqref{eqNSEprojPN} from \eqref{DataAssAlgFuncGalerkin}, we obtain that
 \be\label{eqw}
  \frac{\rd \bw}{\rd t} + \nu A \bw + P_N [ B(\bv_N,\bv_N) - B(\bu,\bu)] = -\beta P_N I_h(\bv_N - \bu).
 \ee
 As in \eqref{rewritenonlinearterms}, we rewrite
 \begin{multline}
  B(\bv_N,\bv_N) - B(\bu,\bu) = \\
  = B(\bw, \bp) + B(\bp,\bw) + B(\bw,\bw) - B(\bp, \bq) - B(\bq,\bp) - B(\bq,\bq).
 \end{multline}

Moreover,
 \begin{eqnarray}\label{rewritefeedback}
  -\beta P_N I_h (\bv_N - \bu) &=& -\beta P_N I_h (\bw) + \beta P_N I_h (\bq) \nonumber \\
  &=& -\beta P_N I_h (\bw) + \beta P_N [I_h (\bq) - \bq].
 \end{eqnarray}

Thus, from \eqref{eqw}-\eqref{rewritefeedback}, we have
 \begin{multline}\label{eqw2}
   \frac{\rd \bw}{\rd t} + \nu A \bw = -P_N[ B(\bw, \bp) + B(\bp,\bw) + B(\bw,\bw) - B(\bp, \bq) - B(\bq,\bp) - B(\bq,\bq)] \\
   -\beta P_N I_h (\bw) + \beta P_N [I_h (\bq) - \bq].
 \end{multline}

Taking the inner product in $L^2$ of \eqref{eqw2} with $A\bw$, yields
 \begin{multline}\label{enstrophyeqw}
  \frac{1}{2} \frac{\rd}{\rd t} \|\bw\|_{H^1}^2 + \nu |A\bw|_{L^2}^2 = -(B(\bw, \bp),A\bw)_{L^2} - (B(\bp,\bw),A\bw)_{L^2} - (B(\bw,\bw),A\bw)_{L^2}\\
  + (B(\bp,\bq),A\bw)_{L^2} + (B(\bq,\bp),A\bw)_{L^2} + (B(\bq,\bq),A\bw)_{L^2} \\
  - \beta\|\bw\|_{H^1}^2 + \beta (\bw - I_h (\bw),A\bw)_{L^2} - \beta (\bq - I_h(\bq), A\bw)_{L^2}.
 \end{multline}

 Using property \eqref{propintP1} of $I_h$, we have
 \begin{eqnarray}\label{est7}
  |\beta(\bw - I_h \bw, A\bw)_{L^2}| &\leq& c_0 \beta h \|\bw\|_{H^1} |A\bw|_{L^2} \nonumber \\
  &\leq& \frac{\beta}{2} \|\bw\|_{H^1}^2 + \frac{c_0^2 \beta h^2}{2} |A\bw|_{L^2}^2 \nonumber \\
  &\leq& \frac{\beta}{2} \|\bw\|_{H^1}^2 + \frac{\nu}{2} |A\bw|_{L^2}^2,
 \end{eqnarray}
 where in the last inequality we used hypothesis \eqref{condh}.

 Now, using property \eqref{propintP1} of $I_h$ and \eqref{boundqH1}, we have
 \begin{eqnarray}\label{est8}
  |\beta (\bq - I_h(\bq), A\bw)_{L^2}| &\leq& c_0 \beta h \|\bq\|_{H^1} |A\bw|_{L^2} \nonumber\\
  &\leq& \beta \|\bq\|_{H^1}^2 + \frac{c_0^2 \beta h^2}{4} |A\bw|_{L^2}^2 \nonumber \\
  &\leq& \beta C_1^2 \frac{L_N^2}{\lambda_{N+1}} + \frac{\nu}{4} |A\bw|_{L^2}^2.
 \end{eqnarray}

 Using, in \eqref{enstrophyeqw}, estimates \eqref{est7}, \eqref{est8} and analogous estimates to \eqref{est1Fourier}-\eqref{est6Fourier}, we obtain that
 \begin{multline}\label{enstrophyineqw1}
  \frac{\rd }{\rd t}\|\bw\|_{H^1}^2 + \frac{\nu}{4} |A\bw|_{L^2}^2 \leq -\beta \|\bw\|_{H^1}^2 \\
  + c M_1 \|\bw\|_{H^1} |A \bw|_{L^2} \left[ 1 + \log\left( \frac{|A \bw|_{L^2}}{\lambda_1^{1/2}\|\bw\|_{H^1}}\right) \right]^{1/2}\\
  + c \|\bw\|_{H^1}^2 |A \bw|_{L^2} \left[ 1 + \log\left( \frac{|A \bw|_{L^2}}{\lambda_1^{1/2}\|\bw\|_{H^1}}\right) \right]^{1/2}
  + c \frac{C_1^4}{\nu} \frac{L_N^6}{\lambda_{N+1}^2} \\
  + c \frac{C_1^2}{\nu} \frac{L_N^4}{\lambda_{N+1}}M_1^2
  + \beta C_1^2 \frac{L_N^2}{\lambda_{N+1}}.
 \end{multline}

 Proceeding analogously as in the proof of Proposition \ref{thmUnifBoundvNH1Fourier}, we obtain that
 \be\label{enstrophyineqw2}
  \frac{\rd }{\rd t}\|\bw\|_{H^1}^2 \leq -\frac{\beta}{2} \|\bw\|_{H^1}^2 + c \frac{C_1^4}{\nu} \frac{L_N^6}{\lambda_{N+1}^2} + c \frac{C_1^2}{\nu} \frac{L_N^4}{\lambda_{N+1}}M_1^2 + \beta C_1^2 \frac{L_N^2}{\lambda_{N+1}},
 \ee
 for all $t \in [t_0, \tilde{t}]$, with $\tilde{t}$ defined as in \eqref{defttildeFourier}.

 Integrating \eqref{enstrophyineqw2} with respect to time from $t_0$ to $t \in [t_0, \tilde{t}]$, we have
 \begin{multline}\label{ineqwH1}
   \|\bw(t)\|_{H^1}^2 \leq \|\bw(t_0)\|_{H^1}^2 \Exp^{-\frac{\beta}{2}(t-t_0)} \\
  + \frac{c}{\beta} \left[ \frac{C_1^4}{\nu} \frac{L_N^6}{\lambda_{N+1}^2} + \frac{C_1^2}{\nu} \frac{L_N^4}{\lambda_{N+1}}M_1^2 + \beta C_1^2 \frac{L_N^2}{\lambda_{N+1}} \right] (1 - \Exp^{-\frac{\beta}{2}(t-t_0)}).
 \end{multline}

 Using hypothesis \eqref{condbeta} with a suitable absolute constant $c$ and similar arguments to the ones used in the proof of Proposition \ref{thmUnifBoundvNH1Fourier}, one obtains that
 \begin{multline}\label{boundresidualterms1}
  \frac{c}{\beta} \left[ \frac{C_1^4}{\nu} \frac{L_N^6}{\lambda_{N+1}^2} + \frac{C_1^2}{\nu} \frac{L_N^4}{\lambda_{N+1}}M_1^2 + \beta C_1^2 \frac{L_N^2}{\lambda_{N+1}}  \right] \leq \\
  \leq \frac{c}{\beta} \left[ \frac{C_1^4}{\nu} \frac{L_m^6}{\lambda_m^2} + \frac{C_1^2}{\nu} \frac{L_m^4}{\lambda_m}M_1^2 + \beta C_1^2 \frac{L_m^2}{\lambda_m} \right] \leq 4 M_1^2.
 \end{multline}

 Using the fact that $\|\bw(t_0)\|_{H^1} \leq 2 M_1$ and inequality \eqref{boundresidualterms1} in \eqref{ineqwH1}, it follows that
 \[
  \|\bw(t)\|_{H^1} \leq 2 M_1, \quad \forall t \in [t_0, \tilde{t}],
 \]
which, by the definition of $\tilde{t}$ in \eqref{defttildeFourier} and the fact that $\bw \in \mC([t_0,\infty),V)$, actually implies that
\[
 \|\bw(t)\|_{H^1} \leq 2 M_1, \quad \forall t \geq t_0.
\]
\end{proof}

Using the results of Proposition \ref{thmUnifBoundvNH1}, Lemma \ref{lemesty} and Proposition \ref{lemboundG}, we can now obtain a uniform estimate in time of $|\bv_N - \bp|_{L^2}$.

\begin{thm}\label{thmestvNPNuIh}
 Let $\bu$ be a solution of \eqref{NSEeqs} satisfying \eqref{unifboundsattractor}-\eqref{boundqH1}, for every $t \geq t_0$. Let $\bv_0 \in B_V(M_1)$, with $M_1$ as in \eqref{unifboundsattractor}. For every $N \in \mathbb{N}$, let $\bv_N$ be the unique solution of \eqref{eqDataAssFourierGalerkin} satisfying $\bv_N(t_0) = P_N \bv_0$. Fix $\alpha \in (1/2,1)$ and consider $m \in \mathbb{N}$ large enough such that
 \begin{multline}\label{condm2}
  \lambda_m \geq \max \left\{ \frac{\lambda_1 \Exp}{2}, c \frac{C_1}{\nu}L_m^2, c \left( \frac{C_1^2}{\nu M_1}\right)^{2/3} L_m^2, c \left(\frac{C_1}{M_1}\right)^2 L_m^2,\right.\\
  \left. \left[ c c_\alpha \left( 1 + \frac{\Exp^{-\alpha}}{1-\alpha}  \right)\frac{|\Omega|^{\alpha - \frac{1}{2}} M_1}{\nu}\right]^{\frac{1}{1-\alpha}} \right\},
 \end{multline}
 where $c_\alpha$ is the constant from \eqref{ineqA-alphadiffBuBv}.

 If $\beta > 0$ is large enough such that
 \be\label{condbeta2}
  \beta \geq \max\left\{ \nu \lambda_m, c \frac{M_1^2}{\nu}\left[ 1 + \log \left(\frac{M_1}{\nu\lambda_1^{1/2}}\right)\right] \right\}
 \ee
 and if $h \geq 0$ is small enough such that
 \be\label{condh2}
   h \leq c \min \left\{ \left(\frac{\nu}{\beta}\right)^{1/2}, \frac{\nu \lambda_m^{1/2}}{\beta} \right\},
 \ee
 then, there exists $\theta = \theta(\beta) \in [0,1)$ and a constant $C = C(\nu, \lambda_1, |\bg|_{L^2},1/h^2)$ such that, for every $N \geq m$, we have
 \be\label{estwL2Ih}
  |\bv_N(t) -\bp(t)|_{L^2} \leq c \theta^{(t - t_0) \nu \lambda_1 -1} |\bv_N(t_0) -\bp(t_0)|_{L^2} + C \frac{L_N}{\lambda_{N+1}^{5/4}}.
 \ee
\end{thm}
\begin{proof}
 We recall equation \eqref{eqwintrosecresults} satisfied by $\bw = \bv_N - \bp$:
 \begin{multline}
  \frac{\rd \bw}{\rd t} + [\nu A P_N + \beta P_N] \bw = - P_N [ B(\bv_N,\bv_N) - B(\bp,\bp)] + P_N G \\
  - \beta P_N P_\sigma [I_h(\bw) -\bw] + \beta P_N P_\sigma I_h(\bq),
 \end{multline}
 where
 \be
  G(t) = B(\bu(t),\bu(t)) - B(\bp(t),\bp(t)), \quad \forall t \geq t_0.
 \ee

 Using Duhamel's formula, it follows that
 \begin{multline}\label{ineqwDuhamel}
  |\bw(t)|_{L^2} \leq |\Exp^{-(t-t_0) (\nu A P_N + \beta P_N)} \bw(t_0)|_{L^2} \\
  + \int_{t_0}^t \left|\Exp^{-(t-\tau) (\nu A P_N + \beta P_N)} P_N [ B(\bv_N(\tau),\bv_N(\tau)) - B(\bp(\tau),\bp(\tau))] \right|_{L^2} \rd \tau \\
  + \int_{t_0}^t \left| \Exp^{-(t-\tau) (\nu A P_N + \beta P_N)} P_N G(\tau) \right|_{L^2} \rd \tau \\
  + \beta \int_{t_0}^t \left| \Exp^{-(t-\tau) (\nu A P_N + \beta P_N)} P_N P_\sigma [ I_h(\bw(\tau)) - \bw(\tau)]\right|_{L^2} \rd \tau \\
  + \beta \int_{t_0}^t  \left| \Exp^{-(t-\tau) (\nu A P_N + \beta P_N)}  P_N P_\sigma I_h(\bq(\tau)) \right|_{L^2} \rd \tau .
 \end{multline}

 The estimates for the first three terms in the right-hand side of \eqref{ineqwDuhamel} now follow by writing
 \be\label{rewritingExpIh}
  \Exp^{-(t-\tau) (\nu A P_N + \beta P_N)} = \Exp^{-(t-\tau) (\nu A P_m + \beta P_m)} + \Exp^{-(t-\tau) (\nu A P_N Q_m + \beta P_N Q_m)}
 \ee
 and proceeding analogously as in the proof of Theorem \ref{thmestvNPNu}, so that
 \be\label{est1Ih}
  |\Exp^{-(t-t_0) (\nu A P_N + \beta P_N)} \bw(t_0)|_{L^2} \leq \Exp^{-(t-t_0)(\nu\lambda_1 + \beta)} |\bw(t_0)|_{L^2};
 \ee
  \begin{multline}\label{est2Ih}
  \int_{t_0}^t \left|\Exp^{-(t-\tau) (\nu A P_N + \beta P_N)} P_N [ B(\bv_N(\tau),\bv_N(\tau)) - B(\bp(\tau),\bp(\tau))] \right|_{L^2} \rd \tau \leq \\
  \leq 3 c_\alpha \frac{|\Omega|^{\alpha - \frac{1}{2}} M_1}{\nu^\alpha} \sup_{\tau \geq t_0} |\bw(\tau)|_{L^2} \left( \frac{(\nu \lambda_m)^\alpha}{\beta} + \frac{\Exp^{-\alpha}}{1-\alpha} \frac{1}{(\nu \lambda_{m+1})^{1-\alpha}} \right),
 \end{multline}
 where we used Proposition \ref{thmUnifBoundvNH1}; and
  \be\label{est3Ih}
  \int_{t_0}^t \left| \Exp^{-(t-\tau) (\nu A P_N + \beta P_N)} P_N G(\tau) \right|_{L^2} \rd \tau \leq c \frac{C_N L_N^2}{\nu},
 \ee
 with $C_N$ given by
 \be\label{defCNIh}
  C_N = C_0\frac{L_N^2}{\lambda_{N+1}^{3/2}} \left( M_1 + C_0 \frac{L_N}{\lambda_{N+1}^{1/2}} \right),
 \ee
 where $C_0$ is defined in \eqref{defC0}.

 In order to estimate the fourth term on the right-hand side of \eqref{ineqwDuhamel}, we use property \eqref{propintP2} of $I_h$ and obtain that
 \begin{multline}\label{est4aIh}
  \beta \int_{t_0}^t \left| \Exp^{-(t-\tau) (\nu A P_N + \beta P_N)} P_N P_\sigma [I_h(\bw(\tau)) - \bw(\tau)]\right|_{L^2} \rd \tau \\
  \leq \frac{\beta}{\nu^{1/2}} \int_{t_0}^t \| \nu^{1/2} A^{1/2} \Exp^{-(t-\tau) (\nu A P_N + \beta P_N)} \|_{\mL(P_N H)} |A^{-1/2} P_\sigma [I_h(\bw) - \bw]|_{L^2} \rd \tau \\
  \leq c_{-1}\frac{\beta h}{\nu^{1/2}} \sup_{\tau \geq t_0} |\bw(\tau)|_{L^2} \int_{t_0}^t \| \nu^{1/2} A^{1/2} \Exp^{-(s-t_0) (\nu A P_N + \beta P_N)} \|_{\mL(P_N H)} \rd s.
 \end{multline}

 Moreover, using again \eqref{rewritingExpIh} and the calculations from the proof of Theorem \ref{thmestvNPNu}, one obtains that
 \be
  \int_{t_0}^t \| \nu^{1/2} A^{1/2} \Exp^{-(s-t_0) (\nu A P_N + \beta P_N)} \|_{\mL(P_N H)} \rd s \leq \frac{(\nu \lambda_m)^{\frac{1}{2}}}{\beta} +
  2\frac{\Exp^{-\frac{1}{2}}}{(\nu\lambda_{m+1})^{\frac{1}{2}}}.
 \ee

 Thus, from \eqref{est4aIh},
 \begin{multline}\label{est4Ih}
  \beta \int_{t_0}^t \left| \Exp^{-(t-\tau) (\nu A P_N + \beta P_N)} P_N [P_\sigma I_h(\bw(\tau)) - \bw(\tau)]\right|_{L^2} \rd \tau \\
  \leq c_{-1}\frac{\beta h}{\nu^{1/2}} \left(\frac{(\nu \lambda_m)^{\frac{1}{2}}}{\beta} + 2\frac{\Exp^{-\frac{1}{2}}}{(\nu\lambda_{m+1})^{\frac{1}{2}}} \right) \sup_{\tau \geq t_0} |\bw(\tau)|_{L^2}
 \end{multline}

 Finally, for the last term in the right-hand side of \eqref{ineqwDuhamel}, we use property \eqref{propintP3} of $I_h$ and obtain that
 \begin{multline}\label{est5aIh}
  \beta \int_{t_0}^t \left| \Exp^{-(t-\tau) (\nu A P_N + \beta P_N)} P_N P_\sigma I_h(\bq(\tau)) \right|_{L^2} \rd \tau \leq \\
  \leq \widetilde{c_0} \frac{|\Omega|^{3/4}}{h^2 \lambda_{N+1}^{1/4}} \beta \int_{t_0}^t \|\Exp^{-(t-\tau) (\nu A P_N + \beta P_N)} \|_{\mL(P_N H)} |\bq(\tau)|_{L^2} \rd \tau.
 \end{multline}

 From \eqref{boundqL2}, we have that
 \be\label{ineqQL2Ih}
  |\bq(\tau)|_{L^2} \leq C_0 \frac{L_N}{\lambda_{N+1}}.
 \ee

 Thus, from \eqref{est5aIh},
 \begin{multline}\label{est5bIh}
  \beta \int_{t_0}^t \left| \Exp^{-(t-\tau) (\nu A P_N + \beta P_N)} P_N P_\sigma I_h(\bq(\tau)) \right|_{L^2} \rd \tau \leq \\
  \leq \widetilde{c_0} \frac{C_0 |\Omega|^{3/4}\beta}{h^2} \frac{L_N}{\lambda_{N+1}^{5/4}} \int_{t_0}^t \|\Exp^{-(t-\tau) (\nu A P_N + \beta P_N)} \|_{\mL(P_N H)} \rd \tau.
 \end{multline}

 Since
 \begin{multline}
  \int_{t_0}^t \|\Exp^{-(t-\tau) (\nu A P_N + \beta P_N)} \|_{\mL(P_N H)} \rd \tau = \int_{t_0}^t \max_{1 \leq j \leq N} \Exp^{-(t-\tau)(\nu \lambda_j + \beta)} \rd \tau \\
  = \int_{t_0}^t \Exp^{-(t-\tau)(\nu \lambda_1 + \beta)} \rd \tau = \frac{1 - \Exp^{-(t-t_0)(\nu \lambda_1 + \beta)}}{\nu \lambda_1 + \beta} \leq \frac{1}{\beta},
 \end{multline}
 then, from \eqref{est5bIh},
 \be\label{est5Ih}
  \beta \int_{t_0}^t \left| \Exp^{-(t-\tau) (\nu A P_N + \beta P_N)} P_N P_\sigma I_h(\bq(\tau)) \right|_{L^2} \rd \tau \leq \widetilde{c_0} \frac{C_0 |\Omega|^{3/4}}{h^2} \frac{L_N}{\lambda_{N+1}^{5/4}}.
 \ee

 Now, using estimates \eqref{est1Ih}-\eqref{est3Ih}, \eqref{est4Ih} and \eqref{est5Ih} in \eqref{ineqwDuhamel}, we obtain that
 \begin{multline}
  |\bw(t)|_{L^2} \leq \Exp^{-(t-t_0)(\nu\lambda_1 + \beta)} |\bw(t_0)|_{L^2} + \\
  + \left[ 3c_\alpha \frac{|\Omega|^{\alpha - \frac{1}{2}} M_1}{\nu^\alpha} \left( \frac{(\nu \lambda_m)^\alpha}{\beta} + \frac{\Exp^{-\alpha}}{1-\alpha} \frac{1}{(\nu \lambda_{m+1})^{1-\alpha}} \right) \right.\\
  \left. + c_{-1}\frac{\beta h}{\nu^{1/2}} \left(\frac{(\nu \lambda_m)^{\frac{1}{2}}}{\beta} + 2\frac{\Exp^{-\frac{1}{2}}}{(\nu\lambda_{m+1})^{\frac{1}{2}}} \right) \right]  \sup_{\tau \geq t_0} |\bw(\tau)|_{L^2} + c \frac{C_N L_N^2}{\nu} \\
  + \widetilde{c_0} \frac{C_0 |\Omega|^{3/4}}{h^2} \frac{L_N}{\lambda_{N+1}^{5/4}}.
 \end{multline}

 Hence, using that $\beta \geq \nu \lambda_m$ and the definition of $C_N$ in \eqref{defCNIh},
 \begin{multline}\label{est6Ih}
  |\bw(t)|_{L^2} \leq \Exp^{-(t-t_0)(\nu\lambda_1 + \beta)} |\bw(t_0)|_{L^2} + \\
  + \left[ c c_\alpha \left( 1 + \frac{\Exp^{-\alpha}}{1-\alpha}\right) \frac{|\Omega|^{\alpha - \frac{1}{2}} M_1}{\nu \lambda_m^{1-\alpha}} +  c_{-1}\frac{\beta h}{\nu \lambda_m^{1/2}} \right]  \sup_{\tau \geq t_0} |\bw(\tau)|_{L^2}
  + C \frac{L_N}{\lambda_{N+1}^{5/4}}.
 \end{multline}

 Using hypotheses \eqref{condm2} and \eqref{condh2} with suitable absolute constants $c$, we obtain that
 \be\label{defgammaIh}
  \gamma = c c_\alpha \left( 1 + \frac{\Exp^{-\alpha}}{1-\alpha}\right) \frac{|\Omega|^{\alpha - \frac{1}{2}} M_1}{\nu \lambda_m^{1-\alpha}} +  c_{-1}\frac{\beta h}{\nu \lambda_m^{1/2}} < 1
 \ee
 and
 \be
  \theta = \Exp^{-\frac{\nu \lambda_1 + \beta}{\nu \lambda_1}} + \frac{\gamma}{1 - \gamma} < 1.
 \ee

 Therefore, \eqref{estwL2Ih} follows from \eqref{est6Ih} and Lemma \ref{lemesty} with $a=1$, $b = \nu\lambda_1 + \beta$, $\gamma$ as given in \eqref{defgammaIh} and $\varepsilon = C L_N/\lambda_{N+1}^{5/4}$.
\end{proof}

The result of Theorem \ref{thmestvNPNuIh} now yields, as in \eqref{ordererrorSGML2}-\eqref{ordererrorSGMH1}, an estimate of the error associated to the Galerkin approximation of \eqref{DataAssAlg} relative to the reference solution $\bu$ of \eqref{NSEeqs}, in the general case of an interpolant operator satisfying properties \eqref{propintP1}-\eqref{propintP3}.

\begin{cor}\label{corerrorSGMIh}
 Assume the hypotheses from Theorem \ref{thmestvNPNuIh}. Then, there exists $T = T(\nu,\lambda_1,|\bg|_{L^2},N) \geq t_0$ such that, for every $N \geq m$,
 \be
  \sup_{t\geq T} |\bv_N(t) - \bu(t)|_{L^2} \leq C\frac{L_N}{\lambda_{N+1}},
 \ee
 and
 \be
  \sup_{t \geq T} \|\bv_N(t) - \bu(t)\|_{H^1} \leq C\frac{L_N}{\lambda_{N+1}^{1/2}},
 \ee
 where $C$ is a constant depending on $\nu$, $\lambda_1$, $|\bg|_{L^2}$ and $1/h^2$, but independent of $N$.
\end{cor}

Finally, we now obtain an estimate of the error committed when applying the Postprocessing Galerkin method to system \eqref{DataAssAlg}, in order to obtain an approximation of the reference solution $\bu$ of \eqref{NSEeqs}, in the case of an interpolant operator satisfying properties \eqref{propintP1}-\eqref{propintP3}. The result shows that the convergence rate of the Postprocessing Galerkin method in this case, although not as good as the one obtained in Theorem \ref{thmerrorPPGM}, is still better than the convergence rate of the standard Galerkin method. The proof follows immediately from the result of Theorem \ref{thmestvNPNuIh} and \eqref{distuPhi1L2b}-\eqref{inversePoincarevNPNu}.

\begin{thm}\label{thmerrorPPGMIh}
Assume the hypotheses from Theorem \ref{thmestvNPNuIh}, with $\bu$ satisfying, in addition, \eqref{distuPhi1L2aa} and \eqref{distuPhi1H1aa}, for every $t \geq t_0$. Then, there exists $T = T(\nu,\lambda_1,|\bg|_{L^2},N)$ $\geq t_0$ such that, for every $N \geq m$,
\be
 \sup_{t \geq T}|[\bv_N(t) + \Phi_1(\bv_N(t))] - \bu(t)|_{L^2} \leq C \frac{L_N}{\lambda_{N+1}^{5/4}},
\ee
and
\be
 \sup_{t \geq T}\|[\bv_N(t) + \Phi_1(\bv_N(t))] - \bu(t)\|_{H^1} \leq C \frac{L_N}{\lambda_{N+1}^{3/4}},
\ee
where $C$ is a constant depending on $\nu$, $\lambda_1$, $|\bg|_{L^2}$ and $1/h^2$, but independent of $N$.
\end{thm}

\begin{rmk}
 We emphasize that the main purpose of the postprocessing step applied to the Galerkin method is to improve the accuracy of the numerical approximation of $\bv$, solution of \eqref{DataAssAlg}, and thus $\bu$, solution of \eqref{NSEeqs}. The fact that the numerical approximation of $\bv$ given by the Postprocessing Galerkin method yields a uniform in time error estimate is actually due to the fact that the Galerkin approximation $\bv_N$ of $\bv$ yields a uniform in time error estimate. Indeed, the latter is valid for an even more general class of interpolant operators than the one considered in subsection \ref{subsecAnotherClassIntOp}. Namely, the family of operators $I_h: H^1(\Omega)^2 \to L^2(\Omega)^2$ which are only required to satisfy property \eqref{propintP1}; and also the family of operators $I_h: H^2(\Omega)^2 \to L^2(\Omega)^2$ satisfying (see \cite{AzouaniOlsonTiti2014})
 \[
  \| \varphi - I_h (\varphi)\|_{H^2} \leq c_1 h \|\varphi\|_{H^1} + c_2 h^2 \|\varphi\|_{H^2} \quad \forall \varphi \in H^2(\Omega)^2,
 \]
where $c_1$ and $c_2$ are positive constants, and $\|\cdot\|_{H^2}$ denotes the usual Sobolev norm of the space $H^2(\Omega)^2$. A physically relevant example of interpolant operator of this latter type is given by measurements at a finite set of nodal points in $\Omega$.

It is not difficult to show that (using, in particular, similar ideas from the proof of Proposition \ref{thmUnifBoundvNH1}), under the appropriate conditions on the parameters $\beta$ and $h$ and for both types of interpolant operators, there exists $T = T(\nu, \lambda_1, |\bg|_{L^2},N) \geq t_0$ large enough such that
\[
 \sup_{t \geq T} \|\bv_N(t) - \bu(t)\|_{H^1} \leq C \frac{L_N^2}{\lambda_{N+1}^{1/2}},
\]
where $C$ is a constant depending on $\nu$, $\lambda_1$ and $|\bg|_{L^2}$, but independent of $N$. Moreover, for the former class of interpolant operators, one can also show that
\[
 \sup_{t \geq T} |\bv_N(t) - \bu(t)|_{L^2} \leq C \frac{L_N}{\lambda_{N+1}},
\]
where, again, $T = T(\nu, \lambda_1, |\bg|_{L^2},N) \geq t_0$ and $C = C(\nu, \lambda_1, |\bg|_{L^2})$.
\end{rmk}

\section{Acknowledgments}

The authors would like to thank Prof. Ciprian Foias for useful and enlightening discussions during the development of this work. The work of C. F. M. was supported by the ONR grant N00014-15-1-2333. The work of E. S. T. was supported in part by the ONR grant N00014-15-1-2333 and the NSF grants DMS-1109640 and DMS-1109645.

\setcounter{equation}{0}
\appendix

\section*{Appendix}
\renewcommand{\thesection}{A}
\renewcommand{\theequation}{{A.}\arabic{equation}}

The aim of this section is to show that the example of interpolant operator given by local averages over finite volume elements (see, e.g.,  \cite{FoiasTiti1991,JonesTiti1992,JonesTiti1993}), assuming periodic boundary conditions, satisfies properties \eqref{propintP1}-\eqref{propintP3} considered in subsection \ref{subsecAnotherClassIntOp}.

Let $\Omega = (0,L)\times(0,L) \subset \mathbb{R}^2$ be a basic domain of periodicity, and consider a partition of $\Omega$ into $K$ squares with sides of length $h = L / \sqrt{K}$. Let
\[
 \Lambda = \{(j,l) \in \mathbb{N}^2: 1 \leq j,l \leq \sqrt{K}\}
\]
and, for every $\alpha = (j,l) \in \Lambda$, let $Q_\alpha$ be the volume element given by the square
\[
 Q_\alpha = [(j-1)h,jh) \times [(l-1)h, lh).
\]

Consider the interpolant operator $I_h: L^2(\Omega)^2 \to L^2(\Omega)^2$ given by
\be\label{defIh}
 I_h(\varphi) = \sum_{\alpha \in \Lambda} \overline{\varphi_\alpha} \chi_{Q_\alpha}, \quad \forall \varphi \in L^2(\Omega)^2,
\ee
where $\overline{\varphi_\alpha}$ is the local average of $\varphi$ over the volume element $Q_\alpha$, i.e.
\be
\overline{\varphi_\alpha} = \frac{1}{|Q_\alpha|} \int_{Q_\alpha} \varphi(y) \rd y .
\ee

The fact that $I_h$ defined in \eqref{defIh} satisfies property \eqref{propintP1} follows from the calculations in \cite[Appendix]{JonesTiti1992}. Thus, it only remains to verify properties \eqref{propintP2} and \eqref{propintP3}. In fact, we show that this particular example of $I_h$ sastisfies a stronger property than \eqref{propintP3}, with respect to the $(L^\infty(\Omega))^2$-norm.

Notice that, in the present case, $|\Omega| = L^2$.

\begin{prop}\label{propAppendix}
Let $I_h: L^2(\Omega)^2 \to L^2(\Omega)^2$ be the operator defined by \eqref{defIh}. Then, it holds:
\begin{enumerate}[(i)]
\item\label{propintP2App} There exists a positive constant $c_{-1}$ such that 
 \be
 \|\varphi - I_h(\varphi)\|_{H^{-1}} \leq c_{-1} h |\varphi|_{L^2}, \quad \forall \varphi \in L^2(\Omega)^2.
\ee
 \item\label{propintP3App} There exists a positive constant $\widetilde{c_0}$ such that
 \be\label{ineqpropintP3App}
  \|I_h(\bq)\|_{L^{\infty}} \leq c \frac{L^{1/2}}{h^2 \lambda_{N+1}^{1/4}} |\bq|_{L^2}, \quad \forall \bq \in Q_N H.
 \ee
 Consequently,
 \be
   |I_h(\bq)|_{L^2} \leq \widetilde{c_0} \frac{L^{3/2}}{h^2 \lambda_{N+1}^{1/4}} |\bq|_{L^2}, \quad \forall \bq \in Q_N H.
 \ee
\end{enumerate}
\end{prop}
\begin{proof}
From its definition, it follows immediately that $I_h$ is a symmetric operator, i.e.
\be\label{Ihsymmetric}
(I_h(\varphi), \psi) = (\varphi, I_h(\psi)), \quad \forall \psi \in L^2(\Omega)^2.
\ee
Thus, using \eqref{Ihsymmetric} and property \eqref{propintP1}, we obtain that
\begin{multline}
\|\varphi - I_h(\varphi)\|_{H^{-1}} = \sup_{\stackrel{\psi \in H_0^1(\Omega)^2}{\|\psi\|_{H^1}=1}} |(\varphi - I_h(\varphi), \psi)| =  \sup_{\stackrel{\psi \in H_0^1(\Omega)^2}{\|\psi\|_{H^1}=1}} |(\varphi,\psi - I_h(\psi))| \\
\leq \sup_{\stackrel{\psi \in H_0^1(\Omega)^2}{\|\psi\|_{H^1}=1}} c_0 h |\varphi|_{L^2} \|\psi\|_{H^1} = c_0 h |\varphi|_{L^2},
\end{multline}
which proves that \eqref{propintP2App} is satisfied with $c_{-1} = c_0$.

Now let us prove \eqref{propintP3App}.
Let $\bq \in Q_N H$ and consider its Fourier expansion, given by
\be\label{expansionq}
 \bq(\by) = \sum_{|k| \geq \kappa_N} \hat{\bu}_{k} \Exp^{2 \pi i\frac{k}{L}\cdot \by}, \quad \forall \by \in \Omega,
\ee
where
\be
\kappa_N = \frac{L}{2\pi} \lambda_{N+1}^{1/2}.
\ee

From \eqref{expansionq} and the definition of $I_h$ in \eqref{defIh}, we have that
\[
 I_h(\bq)(\bx) = \sum_{\alpha \in \Lambda} \sum_{|k| \geq \kappa_N} \frac{1}{|Q_\alpha|} \hat{\bu}_{k} \left( \int_{Q_\alpha} \Exp^{2 \pi i\frac{k}{L}\cdot \by} \rd \by \right) \chi_{Q_\alpha}(\bx).
\]

Thus,
\be\label{estabsvalueIhqx}
 |I_h(\bq)(\bx)| \leq  \sum_{\alpha \in \Lambda} \sum_{|k| \geq \kappa_N} \frac{1}{|Q_\alpha|} |\hat{\bu}_{k}| \left| \int_{Q_\alpha} \Exp^{2 \pi i\frac{k}{L}\cdot \by} \rd \by\right| \chi_{Q_\alpha}(\bx)
 = (S_1 + S_2 + S_3)(\bx),
\ee
where
\[
 S_1(\bx) = \sum_{\alpha \in \Lambda} \sum_{\stackrel{|k| \geq \kappa_N}{k_1 = 0}} \frac{1}{|Q_\alpha|} |\hat{\bu}_{k}| \left| \int_{Q_\alpha} \Exp^{2 \pi i\frac{k}{L} \cdot \by} \rd \by\right| \chi_{Q_\alpha}(\bx),
\]
\[
 S_2(\bx) = \sum_{\alpha \in \Lambda} \sum_{\stackrel{|k| \geq \kappa_N}{k_2 = 0}} \frac{1}{|Q_\alpha|} |\hat{\bu}_{k}| \left| \int_{Q_\alpha} \Exp^{2 \pi i\frac{k}{L}\cdot \by} \rd \by\right| \chi_{Q_\alpha}(\bx),
\]
\[
 S_3(\bx) = \sum_{\alpha \in \Lambda} \sum_{\stackrel{|k| \geq \kappa_N}{k_1 \neq 0, k_2 \neq 0}} \frac{1}{|Q_\alpha|} |\hat{\bu}_{k}| \left| \int_{Q_\alpha} \Exp^{2 \pi i\frac{k}{L} \cdot \by} \rd \by\right| \chi_{Q_\alpha}(\bx).
\]

Notice that
\begin{multline}\label{estS1}
 S_1(\bx) = \sum_{\alpha \in \Lambda} \sum_{\stackrel{|k| \geq \kappa_N}{k_1 = 0}} \frac{1}{|Q_\alpha|} |\hat{\bu}_{k}| \left| \int_{(l-1)h}^{lh} \int_{(j-1)h}^{jh} \Exp^{2 \pi i\frac{k_2}{L} y_2} \rd y_1 \rd y_2 \right| \chi_{Q_\alpha}(\bx) \\
 = \sum_{\alpha \in \Lambda} \sum_{\stackrel{|k| \geq \kappa_N}{k_1 = 0}} \frac{1}{h^2} |\hat{\bu}_{k}| h \left| \frac{L}{2 \pi i k_2}\Exp^{2 \pi i\frac{k_2}{L} l h}(1 - \Exp^{- 2 \pi i\frac{k_2}{L} h}) \right|  \chi_{Q_\alpha}(\bx) \\
 \leq \frac{L}{\pi h} \sum_{\stackrel{|k| \geq \kappa_N}{k_1 = 0}}  |\hat{\bu}_{k}| \frac{1}{|k_2|} \left( \sum_{\alpha \in \Lambda} \chi_{Q_\alpha}(\bx) \right) = \frac{L}{\pi h} \sum_{\stackrel{|k| \geq \kappa_N}{k_1 = 0}}  |\hat{\bu}_{k}| \frac{1}{|k_2|} \\
 \leq \frac{L}{\pi h} \left(\sum_{\stackrel{|k| \geq \kappa_N}{k_1 = 0}} |\hat{\bu}_{k}|^2 \right)^{1/2} \left(\sum_{\stackrel{|k| \geq \kappa_N}{k_1 = 0}} \frac{1}{|k_2|^2} \right)^{1/2} \leq \frac{L}{\pi h} \frac{|\bq|_{L^2}}{|\Omega|^{1/2}} \left(\sum_{\stackrel{|k| \geq \kappa_N}{k_1 = 0}} \frac{1}{|k_2|^2} \right)^{1/2}\\
  = \frac{1}{\pi h} |\bq|_{L^2} \left(\sum_{\stackrel{|k| \geq \kappa_N}{k_1 = 0}} \frac{1}{|k_2|^2} \right)^{1/2} \leq   \frac{c}{h} |\bq|_{L^2} \frac{1}{\kappa_N^{1/2}} \leq \frac{c}{hL^{1/2}} |\bq|_{L^2} \frac{1}{\lambda_{N+1}^{1/4}}.
\end{multline}

Analogously,
\be\label{estS2}
 S_2(\bx) \leq \frac{c}{hL^{1/2}} |\bq|_{L^2} \frac{1}{\lambda_{N+1}^{1/4}}.
\ee

Moreover,
\begin{multline}\label{estS3}
 S_3(\bx) = \sum_{\alpha \in \Lambda} \sum_{\stackrel{|k| \geq \kappa_N}{k_1 = 0}} \frac{1}{|Q_\alpha|} |\hat{\bu}_{k}| \left| \int_{(l-1)h}^{lh} \int_{(j-1)h}^{jh} \Exp^{2 \pi i\frac{k_1}{L} y_1} \Exp^{2 \pi i\frac{k_2}{L} y_2} \rd y_1 \rd y_2 \right| \chi_{Q_\alpha}(\bx) \\
 \leq \frac{L^2}{\pi^2 h^2} \sum_{\stackrel{|k| \geq \kappa_N}{k_1 \neq 0, k_2 \neq 0}} |\hat{\bu}_{k}| \frac{1}{|k_1| |k_2|} \leq \frac{L^2}{\pi^2 h^2} \frac{|\bq|_{L^2}}{|\Omega|^{1/2}} \left[ \sum_{\stackrel{|k| \geq \kappa_N}{k_1 \neq 0, k_2 \neq 0}} \frac{1}{k_1^2 k_2^2} \right]^{1/2} \\
 \leq  \frac{L}{\pi^2 h^2} |\bq|_{L^2} \left[ \sum_{|k_1| \geq \frac{\kappa_N}{2}, |k_2| \geq 1} \frac{1}{k_1^2 k_2^2} + \sum_{|k_2| \geq \frac{\kappa_N}{2}, |k_1| \geq 1} \frac{1}{k_1^2 k_2^2}\right]^{1/2} \\
 \leq \frac{L}{\pi^2 h^2} |\bq|_{L^2} \left[ \left( \sum_{|k_1| \geq \frac{\kappa_N}{2}} \frac{1}{k_1^2}\right) \left( \sum_{|k_2| \geq 1} \frac{1}{k_2^2}\right) + \left( \sum_{|k_2| \geq \frac{\kappa_N}{2}} \frac{1}{k_2^2}\right) \left( \sum_{|k_1| \geq 1} \frac{1}{k_1^2}\right) \right]^{1/2}\\
 \leq c \frac{L}{h^2} |\bq|_{L^2} \frac{1}{\kappa_N^{1/2}} \leq c \frac{L^{1/2}}{h^2} |\bq|_{L^2} \frac{1}{\lambda_{N+1}^{1/4}}.
\end{multline}

From \eqref{estabsvalueIhqx}-\eqref{estS3}, we obtain that
\be\label{estIhLinfty}
 |I_h(\bq)(x)| \leq c \frac{L^{1/2}}{h^2} |\bq|_{L^2} \frac{1}{\lambda_{N+1}^{1/4}}, \quad \forall \bx \in \Omega,
\ee
which proves \eqref{ineqpropintP3App}.
\end{proof}



\begin{thebibliography}{99}

\bibitem{AltafTitiKnioZhaoMcCabeHoteit2015} M. U. Altaf, E. S. Titi, T. Gebrael, O. Knio, L. Zhao, M. F. McCabe and I. Hoteit, \emph{Downscaling the 2D B\'enard convection equations using continuous data assimilation}, arXiv:1512.04671 [math.OC].

\bibitem{AzouaniOlsonTiti2014} A. Azouani, E. Olson and E. S. Titi, \emph{Continuous Data Assimilation Using General Interpolant Observables}, J. Nonlinear Sci., 24 (2014), pp. 277--304.

\bibitem{AzouaniTiti2014} A. Azouani and E.S. Titi, \emph{Feedback control of nonlinear dissipative systems by  finite determining   parameters - a reaction-diffusion paradigm}, Evol. Equ. Control Theory, 3(4) (2014), pp. 579--594.

\bibitem{BartuccelliGibbon2011} M. Bartuccelli and J. Gibbon, \emph{Sharp constants in the Sobolev embedding theorem and a derivation of the Brezis-Gallouet interpolation inequality}, J. Math. Phys., 52 (2011).

\bibitem{BessaihOlsonTiti2015} H. Bessaih, E. Olson and E. S. Titi, \emph{Continuous data assimilation with stochastically noisy data}, Nonlinearity, 28 (2015), no. 3, pp. 729--753.

\bibitem{BlomkerLawStuartZygalakis2013} D. Bl\"omker, K. Law, A. M. Stuart and K. C. Zygalakis, \emph{Accuracy and stability of the continuous-time 3DVAR filter for the Navier–Stokes equation}, Nonlinearity, 26 (2013), pp. 2193--2219.

\bibitem{BrezisGallouet1980} H. Br\'ezis and T. Gallouet, \emph{Nonlinear Schr\"odinger evolution equations}, Nonlinear Anal. TMA, 4 (1980), pp. 677--681.

\bibitem{CharneyHalemJastrow1969} J. Charney, J. Halem, and M. Jastrow, \emph{Use of incomplete historical data to infer the present state of the atmosphere}, J. Atmos. Sci., 26 (1969), pp. 1160--1163.

\bibitem{bookcf1988} P. Constantin and C. Foias, \emph{Navier-Stokes Equations}, Chicago Lectures in Mathematics, University of Chicago Press, Chicago, IL, 1988.

\bibitem{DevulderMarionTiti1993} C. Devulder, M. Marion and E. S. Titi, \emph{On the rate of convergence of the Nonlinear Galerkin methods}, Math. Comp., 60 (1993), no. 202, pp. 495--514.

\bibitem{FarhatJollyTiti2015} A. Farhat, M. S. Jolly and E. S. Titi, \emph{Continuous data assimilation for the 2D B\'enard convection through velocity measurements alone}, Phys. D, 303 (2015), pp. 59--66.

\bibitem{FarhatLunasinTiti2016-2DNSE} A. Farhat, E. Lunasin and E.S. Titi, \emph{Abridged continuous data assimilation for the 2D Navier-Stokes equations utilizing measurements of only one component of the velocity field}, J. Math. Fluid Mech., 18(1) (2016), pp. 1--23.

\bibitem{FarhatLunasinTiti2016-3DBenard} A. Farhat, E. Lunasin and E. S. Titi, \emph{Data assimilation algorithm for 3D B\'enard convection in porous media employing only temperature measurements}, J. Math. Anal. Appl., 438(1) (2016), pp. 492--506.

\bibitem{FarhatLunasinTiti2016-2DBenard} A. Farhat, E. Lunasin and E. S. Titi, \emph{Continuous data assimilation algorithm for a 2D B\'enard convection system through horizontal velocity measurements alone}, J. Nonlinear Sci. (to appear), arXiv:1602.00042 [math.AP] (2016).

\bibitem{FarhatLunasinTiti2016-CharneyConj} A. Farhat, E. Lunasin and E.S. Titi, \emph{On the Charney conjecture of data assimilation employing temperature measurements alone: the paradigm of 3D planetary geostrophic model}, arXiv:1608.04770 [math.AP] (2016).

\bibitem{FoiasJollyKevrekidisSellTiti1988} C. Foias, M. S. Jolly, I. G. Kevrekidis, G. R. Sell and E. S. Titi, \emph{On the computation of inertial manifolds}, Phys. Lett. A, 131 (1988), pp. 433--436.

\bibitem{FMRT2001} C. Foias, O. Manley, R. Rosa and R. Temam, \emph{Navier-Stokes Equations and Turbulence}, Encyclopedia of Mathematics and its Applications, Vol. 83. Cambridge University Press, Cambridge, 2001.  

\bibitem{FoiasManleyTemam1988} C. Foias, O. Manley and R. Temam, \emph{Modelling of the interaction of small and large eddies in two dimensional turbulent flows}, ESAIM Math. Model. Numer. Anal., 22 (1988), 93--114.

\bibitem{FoiasManleyTemamTreve1983} C. Foias, O. Manley, R. Temam and Y. Treve, \emph{Asymptotic analysis of the Navier-Stokes equations}, Physica 6D (1983), pp. 157--188.

\bibitem{FoiasMondainiTiti2016} C. Foias, C. F. Mondaini and E. S. Titi, \emph{A Discrete Data Assimilation scheme for the solutions of the 2D Navier-Stokes equations and their statistics}, SIAM J. Appl. Dyn. Syst. (to appear), arXiv:1602.05995v2 [math.AP] (2016).

\bibitem{FoiasSellTemam1988} C. Foias, G. R. Sell and R. Temam, \emph{Inertial manifolds for nonlinear evolutionary equations}, J. Differential Equations, 73 (1988), pp. 309--353.

\bibitem{FoiasSellTiti1989} C. Foias, G. R. Sell and E. S. Titi, \emph{Exponential tracking and approximation of inertial manifolds for dissipative nonlinear equations}, J. Dynam. Differential Equations, 1 (1989), pp. 199--243.

\bibitem{FoiasTiti1991} C. Foias and E. S. Titi, \emph{Determining nodes, finite difference schemes and inertial manifolds}, Nonlinearity, 4 (1991), pp. 135--153.

\bibitem{GarciaNovoTiti1998} B. Garc\'ia-Archilla, J. Novo and E. S. Titi, \emph{Postprocessing the Galerkin method: a novel approach to approximate inertial manifolds}, SIAM J. Numer. Anal., 35 (1998), pp. 941--972.

\bibitem{GarciaNovoTiti1999} B. Garc\'ia-Archilla, J. Novo and E. S. Titi, \emph{An approximate inertial manifolds approach to postprocessing the Galerkin method for the Navier-Stokes equations}, Math. Comp., 68 (1999), pp. 893--911.

\bibitem{GarciaArchillaTiti2000} B. Garc\'ia-Archilla and E. S. Titi, \emph{Postprocessing the Galerkin method: The finite elements case}, SIAM J. Numer. Anal., 37 (2000), pp. 470--499.

\bibitem{GeshoOlsonTiti2015} M. Gesho, E. Olson and E. Titi, \emph{A computational study of a data assimilation algorithm for the two dimensional Navier-Stokes equations}, Commun. Comput. Phys., 19 (2016), no. 4, pp. 1094--1110.

\bibitem{GhilHalemAtlas1978} M. Ghil, M. Halem and R. Atlas, \emph{Time-continuous assimilation of remote-sounding
data and its effect on weather forecasting}, Mon. Weather Rev., 107 (1978), pp. 140--171.

\bibitem{GhilShkollerYangarber1977} M. Ghil, B. Shkoller and V. Yangarber, \emph{A balanced diagnostic system compatible with
a barotropic prognostic model}, Mon. Weather Rev., 105 (1977), pp. 1223--1238.

\bibitem{GrahamSteenTiti1993} M. D. Graham, P. H. Steen and E. S. Titi, \emph{Computational efficiency and approximate inertial manifolds for a B\'enard convection system}, J. Nonlinear Sci., 3 (1993), pp. 153--167.

\bibitem{HaydenOlsonTiti2011} K. Hayden, E. Olson and E. S. Titi, \emph{Discrete data assimilation in the Lorenz and 2D Navier-Stokes equations}, Physica D, 240 (2011), pp. 1416--1425.

\bibitem{JollyKevrekidisTiti1990} M. S. Jolly, I. G. Kevrekidis and E. S. Titi, \emph{Approximate inertial manifolds for the Kuramoto-Sivashinsky equation: analysis and computations}, Phys. D, 44 (1990), pp. 38--60.

\bibitem{JonesTiti1992} D. A. Jones and E. S. Titi, \emph{Determining finite volume elements for the 2D Navier-Stokes equations}, Physica D, 60 (1992), pp. 165--174.

\bibitem{JonesTiti1993} D. A. Jones and E. S. Titi, \emph{Upper bounds on the number of determining modes, nodes, and volume elements for the Navier-Stokes equations}, Indiana Math. J., 42 (1993), pp. 875--887.

\bibitem{LawShuklaStuart2014} K. Law, A. Shukla and A. Stuart, \emph{Analysis of the 3DVAR filter for the partially observed Lorenz'63 model}, Discrete Contin. Dyn. Syst., 34(3) (2014), pp. 1061--1078.

\bibitem{Lieb1983} E. H. Lieb, \emph{Sharp constants in the Hardy-Littlewood-Sobolev and related inequalities}, Ann. of Math., 118 (1983), pp. 349--374.

\bibitem{Luenberger1971} D. Luenberger, \emph{An introduction to observers}, IEEE T. Automat. Contr., 16 (1971), pp. 596--602.

\bibitem{LunasinTiti} E. Lunasin and E. S. Titi., \emph{Finite determining parameters feedback control for distributed nonlinear dissipative systems -- a computational study}, arXiv:1506.03709[math.AP].

\bibitem{MargolinTitiWynne2003} L. G. Margolin, E. S. Titi and S. Wynne, \emph{The Postprocessing Galerkin and Nonlinear Galerkin methods -- a truncation analysis point of view}, SIAM J. Numer. Anal., 41(2) (2003), pp. 695--714.

\bibitem{MarionTemam1989} M. Marion and R. Temam, \emph{Nonlinear Galerkin Methods}, SIAM J. Numer. Anal., 26 (1989), pp. 1139--1157.

\bibitem{Mizohata1973} S. Mizohata, \emph{The Theory of Partial Differential Equations}, Cambridge University Press, Cambridge (1973).

\bibitem{Nijmeijer2001} H. Nijmeijer, \emph{A dynamic control view of synchronization}, Physica D, 154 (2001), pp. 219--228.

\bibitem{Temambook1995} R. Temam, \emph{Navier-Stokes Equations and Nonlinear Functional Analysis}, 2nd ed., CBMS-NSF Regional Conference Series in Applied Mathematics, 66, SIAM, Philadelphia, PA, 1995.

\bibitem{Temambook1997} R. Temam, \emph{Infinite-Dimensional Dynamical Systems in Mechanics and Physics}, 2nd ed., Applied Mathematical Sciences, 68, Springer-Verlag, New York, 1997.

\bibitem{Temambook2001} R. Temam, \emph{Navier-Stokes Equations. Theory and Numerical Analysis}, Studies in Mathematics and its Applications, 3rd edition, North-Holland Publishing Co., Amsterdam-New York, 1984. Reedition in the AMS Chelsea Series, AMS, Providence, 2001.

\bibitem{Thau1973} F. E. Thau, \emph{Observing the state of non-linear dynamic systems}, Int. J. Control, 17 (1973), pp. 471--479.

\bibitem{Titi1987} E. S. Titi, \emph{On a criterion for locating stable stationary solutions to the Navier-Stokes equations}, Nonlinear Anal. TMA, 11 (1987), pp. 1085--1102.

\bibitem{Titi1990} E. S. Titi, \emph{On approximate inertial manifolds to the Navier-Stokes equations}, J. Math. Anal. Appl., 149 (1990), pp. 540--557.


\end{thebibliography}
\end{document}